\documentclass[11pt]{article}
\usepackage{}
\usepackage{amssymb}
\usepackage{amsfonts}
\usepackage{mathrsfs,psfrag,eepic,epsfig}
\usepackage{graphicx}
\usepackage{amsmath}
\usepackage{color}
\usepackage{amsthm}
\usepackage{txfonts}
\usepackage{pifont,bm}
\usepackage{tikz}
\usepackage{epstopdf}
\usepackage[title]{appendix}
\usepackage{enumerate}
\usepackage{lineno,hyperref}
\newtheorem{thm}{Theorem}[section]

\newtheorem{prop}[thm]{Proposition}
\theoremstyle{definition}

\newtheorem{rem}[thm]{Remark}
\newtheorem{assum}[thm]{Assumption}
\newtheorem{rhp}[thm]{Riemann-Hilbert Problem}

\makeatletter
\def\@biblabel#1{[#1]}
\makeatother

\makeatletter \@addtoreset{equation}{section}

\begin{document}
%\begin{CJK*}{GBK}{song}

\begin{titlepage}
\title{\bf{Long-time asymptotics for the Elastic Beam equation in the solitonless region via $\bar{\partial}$ methods
\footnote{%Project supported by the Fundamental Research Fund  for Talents Cultivation Project of the China University of Mining and Technology under grant number YC150003.\protect\\
%\hspace*{3ex} $^{*}$
Corresponding authors.\protect\\
\hspace*{3ex} \emph{E-mail addresses}: ychen@sei.ecnu.edu.cn (Y. Chen)}
}}
\author{Wei-Qi Peng$^{1}$, Yong Chen$^{1,2,*}$\\
%%%%%%%%%%%%%%%%%%%%%%%%%%%%%%%%%%%%%%%%%%%%%%%%%%%%%%%%%%%%%%%%%%%%%%%%%%%%%%%%%%%%%%%%%
%%%%%              以下两行为作者单位
%%%%%%%%%%%%%%%%%%%%%%%%%%%%%%%%%%%%%%%%%%%%%%%%%%%%%%%%%%%%%%%%%%%%%%%%%%%%%%%%%%%%%%%%%
\small \emph{$^{1}$School of Mathematical Sciences, Key Laboratory of Mathematics and} \\
\small \emph{ Engineering Applications(Ministry of Education) \& Shanghai Key Laboratory of PMMP,}\\
\small \emph{  East China Normal University, Shanghai 200241, China}\\
\small \emph{ $^{2}$College of Mathematics and Systems Science, Shandong University }\\
\small \emph{of Science and Technology, Qingdao, 266590, China}
\date{}}

\thispagestyle{empty}
\end{titlepage}
\maketitle

\vspace{-0.5cm}
\begin{center}
\rule{15cm}{1pt}\vspace{0.3cm}

\parbox{15cm}{\small
{\bf Abstract}\\
\hspace{0.5cm}  In this work, we study the Cauchy problem of
the Elastic Beam equation with initial value in weighted Sobolev space $H^{1,1}(\mathbb{R})$ via the $\bar{\partial}$-steepset descent method. Begin with the Lax pair of the Elastic Beam equation, we successfully derive the basic Riemann-Hilbert problem, which can be used to represent the solutions
of the Elastic Beam equation. Then, considering the solitonless region and using the $\bar{\partial}$-steepset descent method, we analyse the long-time asymptotic behaviors of the solutions for the Elastic Beam equation.}

\vspace{0.5cm}
\parbox{15cm}{\small{

\vspace{0.3cm} \emph{Key words:}  Elastic Beam equation; Riemann-Hilbert problem; Long-time asymptotics; $\bar{\partial}$-steepset descent method.\\

 } }
\end{center}
\vspace{0.3cm} \rule{15cm}{1pt} \vspace{0.2cm}

%\linenumbers

\tableofcontents

\section{Introduction}
The present work is devoted to the study of the Cauchy problem for the Elastic Beam (EB) equation\cite{Brunelli-JMP}
\begin{align}\label{1}
\begin{cases}
q_{t}=\left(\frac{q_{xx}}{\left(1+q_{x}^{2}\right)^{\frac{3}{2}}}\right)_{x},\quad (x,t)\in \mathbb{R}\times(0,+\infty),\\
q(x,0)=q_{0}(x)\in H^{1,1}(\mathbb{R})
\end{cases}
\end{align}
by the $\bar{\partial}$-steepest descent methods, where the weighted Sobolev space $H^{1,1}(\mathbb{R})$ is defined as
\begin{align}\label{1.1}
&H^{1,1}(\mathbb{R})=L^{2,1}(\mathbb{R})\cap H^{1}(\mathbb{R}),\notag\\
&L^{2,1}(\mathbb{R})=\left\{\left(1+\left\vert \cdot\right\vert^{2}\right)^{\frac{1}{2}}f\in
L^{2}(\mathbb{R})\right\},\notag\\
&H^{1}(\mathbb{R})=\left\{f\in L^{2}(\mathbb{R})|f'\in L^{2}(\mathbb{R})\right\}.
\end{align}
Eq. \eqref{1} admits a Lax pair and belongs to the Wadati-Konno-Ichikawa(WKI) system \cite{Brunelli-JMP4} and its $x$ derivative can be used to describe the nonlinear transverse oscillation of an elastic beam under the action of tension \cite{Brunelli-JMP5}. Therefore, Eq. \eqref{1} can be simply called the EB equation. Notably, the EB equation and  short pulse equation were appeared on the same
hierarchy equations corresponding to negative  and positive  flows \cite{Brunelli-JMP,Brunelli-BJP10}. In Ref.\cite{Brunelli-JMP}, author has shown that these systems possess the  bi-Hamiltonian structure, and they used recursion to construct infinite series of local and non-local conserved charges and the corresponding hierarchy of equations. Besides, a Lax pair for the EB equation was given and the nonlocal versions of EB equation was introduced by Brunelli\cite{Brunelli-BJP10}. If we take the derivative of both sides of the Eq.\eqref{1} with respect to $x$, and take $u(x,t)=q_{x}(x,t)$, the Eq.\eqref{1} finally can be written as
\begin{align}\label{1.2}
u_{t}=\left(\frac{u_{x}}{\left(1+u_{x}^{2}\right)^{\frac{3}{2}}}\right)_{xx},
\end{align}
which is a real WKI-II equation and is related to the motion of non-stretching plane curves in Euclidean geometry $\textbf{E}^{2}$ and $\textbf{E}^{3}$ \cite{Lin-PD13,Lin-PD14}.  The properties of loop solitons of formula \eqref{1.2} and their correlation with Harry Dym equation are found in Refs. \cite{Lin-PD15,Lin-PD16} by the inverse scattering method.

Although the high nonlinearity of Eq.\eqref{1} makes the study of the equations very difficult, the inverse scattering transform based on the Riemann-Hilbert (RH) problem can be used to solve such integrable equations.  The RH problem is still very active today after decades of development and many successful advances were made in  the domain of integrable systems \cite{Wang16,Wang17,Wang18,Wang19,Wang20,Wang21}. In particular, a series of major breakthroughs have been made in the asymptotic behavior of the solution via applying the RH method \cite{MaWemXiu-1,MaWemXiu-2,MaWemXiu-3,MaWemXiu-4,MaWemXiu-5}. Of which the emergence of nonlinear steepest descent method, also known as the Deift-Zhou method inspired by the classic steepest descent method and Its earlier work \cite{MaWemXiu-6}, was a milestone  for analysing the long time asymptotic behavior of the solution \cite{Wang23}.  Since then, Deift, Venakides and Zhou have further developed this approach \cite{MaWemXiu-8,MaWemXiu-9,MaWemXiu-10}, and a large number of researchers have used the nonlinear steepest descent method to study asymptotic analysis of various nonlinear integrable equations
\cite{Peng-17,Peng-18,Peng-19,Peng-20,Peng-21,Peng-22,Peng-23}. It is worth noting that McLaughlin and  Miller further proposed the $\bar{\partial}$-steepest descent method on the basis of the nonlinear descent method \cite{Tian-wang24,Tian-wang25,Tian-wang26}. Compared with the nonlinear
steepest descent method, the $\bar{\partial}$-steepest descent method avoids the fine estimates of $L_{p}$ estimates of Cauchy projection operators and the initial value is in a larger space.  With the
sustainable development of the $\bar{\partial}$-steepest descent technique, more and more research has been done on it,
containing the defocusing and focusing  NLS equation \cite{Tian-wang27, Li-WKI42}, KdV equation\cite{Tian-wang28},  SP equation\cite{Tian-wang30},  Fokas-Lenells equation\cite{Tian-wang31}, modified Camassa-Holm equation\cite{Tian-wang32}, WKI eqaution \cite{Tian-wang33,Tian-wang34} et al.
In this work, we aim to investigate the long-time asymptotic behavior of the EB equation \eqref{1} with the initial value condition belonging to the weighted Sobolev space $H^{1,1}(\mathbb{R})$ through the $\bar{\partial}$-steepest descent method.

\textbf{Organization of the Rest of the Work}:
In Sect. 2, we carry out spectral analysis by introducing two
kinds of eigenfunctions including near $\lambda=\infty$ and $\lambda=0$. Further, similar to the Refs. \cite{Li-WKI55,Geng-PDLi-WKI55}, through analysing
analyticity, symmetries and asymptotic properties of  eigenfunctions, we can construct the basic RH problem $M(\lambda)$
for the EB equation \eqref{1.1}.
In Sect. 3, according to phase analysis, we perform first deformation of the basic RH problem by introducing the matrix function $\delta(\lambda)$. Then, constructing a matrix function $R^{(2)}(\lambda)$, we present the continuous extensions and derive the mixed $\bar{\partial}$-RH problem, which can be
decomposed  into
a pure RH problem with $\bar{\partial}R^{(2)}=0$ and a pure $\bar{\partial}$-RH problem with $\bar{\partial}R^{(2)}\neq 0$. Furthermore, the pure RH problem can be solved through establishing the local solvable model near the phase point $\pm\lambda_{0}$ and estimating the error function $E(\lambda)$ with a small norm RH problem. The error term can be derived by studying the pure $\bar{\partial}$-RH problem.
Finally, we get the long-time asymptotic behavior for the EB equation \eqref{1}.

\section{Spectral analysis and basic Riemann-Hilbert problem}
In this part, we need to perform the spectral analysis and construct the basic RH problem for
the EB equation which admits the Lax pair
\begin{align}\label{2}
\Psi_{x}=U\Psi,\qquad \Psi_{t}=V\Psi,
\end{align}
where
\begin{align}\label{3}
U=\left(\begin{array}{cc}
    i\lambda  &  i\lambda q_{x}\\
    i\lambda q_{x}  &  -i\lambda\\
\end{array}\right),\qquad  V=\left(\begin{array}{cc}
    A  &  B\\
    C  &  -A\\
\end{array}\right)
\end{align}
and
\begin{align}\label{4}
&A=-4i\lambda^{3}m^{-\frac{1}{2}}, B=i\lambda(q_{xx}m^{-\frac{3}{2}})_{x}-2\lambda^{2} q_{xx}m^{-\frac{3}{2}}-4i\lambda^{3}q_{x}m^{-\frac{1}{2}},\notag\\
&C=i\lambda(q_{xx}m^{-\frac{3}{2}})_{x}+2\lambda^{2} q_{xx}m^{-\frac{3}{2}}-4i\lambda^{3}q_{x}m^{-\frac{1}{2}},\ m=1+q_{x}^{2}.
\end{align}

Firstly, we discuss the eigenfunctions near $\lambda=\infty$ of the Lax pair \eqref{2}. According to the idea provided by Boutet de Monvel
and Shepelsky \cite{Li-WKI53,Li-WKI54,Li-WKI57}, we introduce a transformation to modify the Lax pair
of the EB equation \eqref{1}, given by
\begin{align}\label{5}
\Phi(x,t;\lambda)=G(x,t)\Psi(x,t;\lambda),
\end{align}
where
\begin{align}\label{6}
G(x,t)=\sqrt{\frac{1+\sqrt{m}}{2\sqrt{m}}}\left(\begin{array}{cc}
    1  &  \frac{\sqrt{m}-1}{q_{x}}\\
    -\frac{\sqrt{m}-1}{q_{x}}  &  1\\
\end{array}\right),
\end{align}
then, we get the Lax pair related to $\Phi(x,t;\lambda)$
\begin{align}\label{7}
&\Phi_{x}(x,t;\lambda)=(i\lambda\sqrt{m}\sigma_{3}+U^{\infty})\Phi(x,t;\lambda),\notag\\
&\Phi_{t}(x,t;\lambda)=\left[\left(-4i\lambda^{3}+i\lambda\frac{2mq_{x}q_{xxx}-6q_{xx}^{2}q_{x}^{2}-q_{xx}^{2}}{2m^{3}}\right)\sigma_{3}+V^{\infty}\right]\Phi(x,t;\lambda),
\end{align}
where $\sigma_{3}$ is the third Pauli matrix
\begin{align}\label{8}
& \sigma_{1}=\left(\begin{array}{cc}
    0  &  1\\
    1  &  0\\
\end{array}\right),\quad \sigma_{2}=\left(\begin{array}{cc}
    0  &  -i\\
    i  &  0\\
\end{array}\right),\quad \sigma_{3}=\left(\begin{array}{cc}
    1  &  0\\
    0  &  -1\\
\end{array}\right),\quad U^{\infty}=\left(\begin{array}{cc}
    0  &  \frac{q_{xx}}{2m}\\
    -\frac{q_{xx}}{2m}  &  0\\
\end{array}\right),\notag\\
&V^{\infty}=\frac{i\lambda q_{xx}^{2}}{2m^{3}}\sigma_{3}+\left(\begin{array}{cc}
    0  &  -2\lambda^{2}\frac{q_{xx}}{m^{\frac{3}{2}}}+\frac{i\lambda}{m^{3}}(mq_{xxx}-3q_{x}q_{xx}^{2})+\frac{q_{xt}}{2m}\\
    2\lambda^{2}\frac{q_{xx}}{m^{\frac{3}{2}}}+\frac{i\lambda}{m^{3}}(mq_{xxx}-3q_{x}q_{xx}^{2})-\frac{q_{xt}}{2m}  &  0\\
\end{array}\right).
\end{align}
Since the EB
equation \eqref{1} arrives at the conservation law
\begin{align}\label{9}
\left(\sqrt{m}\right)_{t}=\left(\frac{2mq_{x}q_{xxx}-6q_{xx}^{2}q_{x}^{2}-q_{xx}^{2}}{2m^{3}}\right)_{x},
\end{align}
we can denote $p(x, t; \lambda)$ as
\begin{align}\label{10}
p(x, t; \lambda)=x-\int_{x}^{\infty}\left(\sqrt{m(s,t)}-1\right)\mathrm{d}s-4\lambda^{2}t,
\end{align}
which satisfies
\begin{align}\label{11}
p_{x}(x, t; \lambda)=\sqrt{m},\ p_{t}(x, t; \lambda)=\frac{2mq_{x}q_{xxx}-6q_{xx}^{2}q_{x}^{2}-q_{xx}^{2}}{2m^{3}}-4\lambda^{2},
\end{align}
and
\begin{align}\label{12}
p_{xt}(x, t; \lambda)=p_{tx}(x, t; \lambda).
\end{align}
Through defining $\mu=\Phi e^{-i\lambda P\sigma_{3}}$, the Lax pair \eqref{7} becomes
\begin{align}\label{13}
\mu_{x}(x, t; \lambda)=i\lambda p_{x}(x, t; \lambda)\left[\sigma_{3},\mu(x, t; \lambda)\right]+U^{\infty}\mu(x, t; \lambda),\notag\\
\mu_{t}(x, t; \lambda)=i\lambda p_{t}(x, t; \lambda)\left[\sigma_{3},\mu(x, t; \lambda)\right]+V^{\infty}\mu(x, t; \lambda),
\end{align}
which can be solved by the following Volterra type integral equations
\begin{align}\label{14}
\mu_{\pm}(x, t; \lambda)=\mathbb{I}+\int_{\pm \infty}^{x}e^{i\lambda\left[p\left(x, t; \lambda\right)-p\left(y, t; \lambda\right)\right]\hat{\sigma}_{3}}U^{\infty}(y, t; \lambda)\mu_{\pm}(y, t; \lambda)\mathrm{d}y,
\end{align}
where $e^{\hat{\sigma}_{3}}D=e^{\sigma_{3}}De^{-\sigma_{3}}$. Following the above analysis, we can derive
the properties of the eigenfunctions $\mu_{\pm}(x, t; \lambda)$.
\begin{prop}
 {\rm (Analytic property)} Let $q(x)-q_{0}\in H^{1,1}(\mathbb{R})$, it is easy to find that
 $\mu_{+,1}, \mu_{-,2}$ are analytic in $\mathbb{C}_{+}$ and $\mu_{-,1}, \mu_{+,2}$ are analytic in $\mathbb{C}_{-}$. $\mathbb{C}_{\pm}$ represent the upper and lower complex $\lambda$-plane, and $\mu_{\pm,j}(j=1, 2)$ represent the $j$-th column of $\mu_{\pm}$.
\end{prop}

\begin{prop}
 {\rm (Symmetry property)} The eigenfunctions $\mu_{\pm}(x, t; \lambda)$ posses
the following symmetry relation
\begin{align}\label{15}
\mu_{\pm}(x, t; \lambda)=\overline{\mu_{\pm}}(x, t; -\overline{\lambda})=\sigma_{2}\overline{\mu_{\pm}}(x, t; \overline{\lambda})\sigma_{2},
\end{align}
where the $\overline{D}$ denotes the complex conjugate of $D$.
\end{prop}

\begin{prop}
 {\rm (Asymptotic property for $\lambda\rightarrow\infty$)} The following asymptotic  behavior comes into existence for the eigenfunctions $\mu_{\pm}(x, t; \lambda)$
\begin{align}\label{16}
\mu_{\pm}(x, t; \lambda)=\mathbb{I}+O(\frac{1}{\lambda}), \quad \lambda\rightarrow\infty.
\end{align}
\end{prop}

\subsection{The eigenfunctions near $\lambda=0$}
For the EB equation, we also need to
investigate the eigenfunctions near $\lambda=0$. Since the initial value $q_{0}\in H^{1,1}(\mathbb{R})$,  as $x\rightarrow\pm\infty$, the eigenfunctions $\Psi(x, t; \lambda)$ satisfies the asymptotic behavior
\begin{align}\label{17}
\Psi(x, t; \lambda)\sim e^{i(\lambda x-4\lambda^{3}t)\sigma_{3}},\ x\rightarrow \pm\infty.
\end{align}
Through making a gauge transformation
\begin{align}\label{18}
\mu^{0}(x, t; \lambda)=\Psi(x, t; \lambda)e^{-i(\lambda x-4\lambda^{3}t)\sigma_{3}},
\end{align}
we get Jost solutions $\mu^{0}(x, t; \lambda)$, which tends to $\mathbb{I}$ as $x\rightarrow \pm\infty$. Then, we obtain the new Lax
pair for $\mu^{0}(x, t; \lambda)$, given by
\begin{align}\label{19}
&\mu^{0}_{x}(x, t; \lambda)=i\lambda\left[\sigma_{3},\mu^{0}(x, t; \lambda)\right]+U^{0}\mu^{0}(x, t; \lambda),\notag\\
&\mu^{0}_{t}(x, t; \lambda)=-4i\lambda^{3}\left[\sigma_{3},\mu^{0}(x, t; \lambda)\right]+V^{0}\mu^{0}(x, t; \lambda),
\end{align}
where $U^{0}=i\lambda q_{x}\sigma_{1}, V^{0}=V+4i\lambda^{3}\sigma_{3}$. Then, the Eq. \eqref{19} can be solved by the Volterra
integrals
\begin{align}\label{20}
\mu^{0}_{\pm}(x, t; \lambda)=\mathbb{I}+\int_{\pm \infty}^{x}e^{i\lambda\left(x-y\right)\hat{\sigma}_{3}}U^{0}(y, t; \lambda)\mu^{0}_{\pm}(y, t; \lambda)\mathrm{d}y,
\end{align}
which implies the analytical properties for $\mu_{\pm}^{0}(x, t; \lambda)$.
\begin{prop}
 {\rm (Analytic property)} Let $q(x)-q_{0}\in H^{1,1}(\mathbb{R})$, we find
$\mu^{0}_{+,1}, \mu^{0}_{-,2}$ are analytic in $\mathbb{C}_{+}$ and $\mu^{0}_{-,1}, \mu^{0}_{+,2}$ are analytic in $\mathbb{C}_{-}$.
\end{prop}

\begin{prop}
{\rm (Asymptotic property for $\lambda\rightarrow 0$)} As $\lambda\rightarrow 0$, the asymptotic property of $\mu_{\pm}^{0}(x, t; \lambda)$ arrives
\begin{align}\label{21}
\mu^{0}_{\pm}(x, t; \lambda)=\mathbb{I}+iq(x,t)\sigma_{1}\lambda+O(\lambda^{2}),\ \lambda\rightarrow 0.
\end{align}
\end{prop}

\subsection{The Connection Between $\mu_{\pm}(x, t; \lambda)$ and $\mu^{0}_{\pm}(x, t; \lambda)$, Scattering matrix}
In order to reconstruct the solution $q(x, t)$ from  the basis  RH problem, the connection
between $\mu_{\pm}(x, t; \lambda)$ and $\mu^{0}_{\pm}(x, t; \lambda)$  need to be established. In fact,
according to \eqref{5} and \eqref{18}, the $G^{-1}(x, t)\mu_{\pm}(x, t; \lambda) e^{i\lambda p\sigma_{3}}$ and $\mu^{0}_{\pm}(x, t; \lambda)e^{i(\lambda x-4\lambda^{3}t)\sigma_{3}}$ have following linear relation
\begin{align}\label{22}
\mu_{\pm}(x, t; \lambda)=G(x, t)\mu^{0}_{\pm}(x, t; \lambda)e^{i(\lambda x-4\lambda^{3}t)\sigma_{3}}C_{\pm}(\lambda)e^{-i\lambda p\sigma_{3}},
\end{align}
where $C_{\pm}(\lambda)$ is a matrix-valued function  with $C_{+}=\mathbb{I}, C_{-}=e^{-i\lambda c\sigma_{3}}, c=\int_{-\infty}^{+\infty}(\sqrt{m(s,t)}-1)\mathrm{d}s=c_{+}+c_{-}, c_{+}=\int_{x}^{+\infty}(\sqrt{m(s,t)}-1)\mathrm{d}s, c_{-}=\int_{-\infty}^{x}(\sqrt{m(s,t)}-1)\mathrm{d}s$. As a result, one has
\begin{align}\label{23}
\mu_{\pm}(x, t; \lambda)=G(x, t)\mu^{0}_{\pm}(x, t; \lambda)e^{i\lambda \int_{x}^{\pm\infty}(\sqrt{m(s,t)}-1)\mathrm{d}s \sigma_{3}}.
\end{align}
Obviously, the eigenfunctions $\mu_{+}(x, t; \lambda)$ and $\mu_{-}(x, t; \lambda)$ have following linear relation by a scattering matrix  $S(\lambda)$
\begin{align}\label{24}
\mu_{-}(x, t; \lambda)=\mu_{+}(x, t; \lambda)e^{i\lambda p\hat{\sigma_{3}}}S(\lambda),
\end{align}
where the matrix $S(\lambda)$ has the form
\begin{align}\label{25}
S(\lambda)=\left(\begin{array}{cc}
    \overline{a(\bar{\lambda})}  &  b(\lambda)\\
    -\overline{b(\bar{\lambda})}  &  a(\lambda)\\
\end{array}\right),
\end{align}
of which $a(\lambda)$ can be computed by
\begin{align}\label{26}
a(\lambda)= \det(\mu_{+1},\mu_{-2}),
\end{align}
which implies that
$a(\lambda)$ is analytic in $\mathbb{C}_{+}$.

\begin{assum}\label{Ass1}
In this work, we assume that the initial value $q_{0}(x)$ is selected so that $a(\lambda)$ is non-zero, and in general we can suppose that $q_{0}(x)$ has a small norm.
\end{assum}

\begin{prop}
 {\rm (Asymptotic property for $\lambda\rightarrow\infty$)} The asymptotic behavior for the scattering data  $a(\lambda)$ is
\begin{align}\label{27}
a(\lambda)=1+O(\frac{1}{\lambda}), \quad \lambda\rightarrow\infty.
\end{align}
\end{prop}

\subsection{The basic Riemann-Hilbert problem}
According to the property of the Jost solutions and the scattering data, the sectionally meromorphic matrices can be defined as
\begin{align}\label{28}
\tilde{M}(x, t; \lambda)=\begin{cases}
\tilde{M}_{+}(x, t; \lambda)=\left(\mu_{+1}, \frac{\mu_{-2}}{a(\lambda)}\right),\quad \lambda\in \mathbb{C}_{+}\\
\tilde{M}_{-}(x, t; \lambda)=\left(\frac{\mu_{-1}}{\overline{a(\bar{\lambda})}}, \mu_{+2}\right),\quad \lambda\in \mathbb{C}_{-},
\end{cases}
\end{align}
where $\tilde{M}_{\pm}(x, t; \lambda)=\underset{\epsilon\rightarrow 0}{\lim}\tilde{M}_{\pm}(x, t; \lambda\pm i\epsilon), \lambda\in \mathbb{R}$.
Then combining the Assumption \ref{Ass1}, the above matrix function
$\tilde{M}(x, t; \lambda)$ solves the following matrix RH problem.

\begin{rhp}\label{rhp1}
Find an analysis function $\tilde{M}(x, t; \lambda)$ which satisfies:\\

 $\bullet$ $\tilde{M}(x, t; \lambda)$ is meromorphic in $\mathbb{C}\setminus \mathbb{R}$;\\

 $\bullet$ $\tilde{M}_{+}(x, t; \lambda)=\tilde{M}_{-}(x, t; \lambda)\tilde{J}(x, t; \lambda), \lambda\in \mathbb{R}$, where
\begin{align}\label{29}
\tilde{J}(x, t; \lambda)=\left(\begin{array}{cc}
    1  &  r(\lambda)e^{2i\lambda p}\\
    \overline{r(\lambda)}e^{-2i\lambda p}  &  1+\left\vert r(\lambda) \right\vert^{2}\\
\end{array}\right),
\end{align}

 with $r(\lambda)=\frac{b(\lambda)}{a(\lambda)}$;

$\bullet$ $\tilde{M}(x, t; \lambda)=\mathbb{I}+O(\frac{1}{\lambda})$ as $\lambda\rightarrow\infty$;\\

$\bullet$ $\tilde{M}(x, t; \lambda)=G(x,t)\left[\mathbb{I}+(ic_{+}\sigma_{3}+iq(x,t)\sigma_{1})\lambda+O(\lambda^{2})\right]$ as $\lambda\rightarrow 0$ .
\end{rhp}

\begin{rem}\label{rem1}
 According to the results of the Zhou's vanishing lemma and
Liouville's theorem, we know that the solution of RH problem \ref{rhp1} exists and is unique.
\end{rem}

Next, Our goal is to reconstruct the solution
$q(x, t)$. Following the idea raised by Boutet de Monvel and Shepelsky \cite{Li-WKI57,Li-WKI58}, we introduce a new scale
\begin{align}\label{30}
y(x,t)=x-c_{+}(x,t).
\end{align}
In this scale, the original RH problem $\tilde{M}(x, t; \lambda)$ can be
further defined into
\begin{align}\label{31}
\tilde{M}(x, t; \lambda)=M(y(x, t), t; \lambda),
\end{align}
and the new function $M(x, t; \lambda)$ satisfies the following matrix RH problem:

\begin{rhp}\label{rhp2}
Find an analysis function $M(y, t; \lambda)$ which satisfies:\\

 $\bullet$ $M(y, t; \lambda)$ is meromorphic in $\mathbb{C}\setminus \mathbb{R}$;\\

 $\bullet$ $M_{+}(y, t; \lambda)=M_{-}(y, t; \lambda)J(y, t; \lambda), \lambda\in \mathbb{R}$, where
\begin{align}\label{32}
J(y, t; \lambda)=\left(\begin{array}{cc}
    1  &  r(\lambda)e^{-2it\theta}\\
    \overline{r(\lambda)}e^{2it\theta}  &  1+\left\vert r(\lambda) \right\vert^{2}\\
\end{array}\right),
\end{align}

 with $\theta(y,t; \lambda)=4\lambda^{3}-\frac{y}{t}\lambda$;

$\bullet$ $M(y, t; \lambda)=\mathbb{I}+O(\frac{1}{\lambda})$ as $\lambda\rightarrow\infty$.
\end{rhp}

\begin{prop}
Based on the Remark \ref{rem1}, we know that the RH problem \ref{rhp2}
admits a unique solution. Furthermore, the solution $q(x, t)$ of the Cauchy problem \eqref{1} can be
represented by the solution of RH problem \ref{rhp2} in parametric form, given by
\begin{align}\label{33}
&q(x, t)=q(y(x, t), t)=\lim_{\lambda\rightarrow 0}\frac{\left[M^{-1}(y, t; 0)M(y, t; \lambda)\right]_{12}}{i\lambda},\notag\\
&x(y,t)=y+\lim_{\lambda\rightarrow 0}\frac{\left[M^{-1}(y, t; 0)M(y, t; \lambda)\right]_{11}-1}{i\lambda}.
\end{align}
\end{prop}

\section{Asymptotics in oscillating region: $y>0, \left\vert \frac{y}{t}=O(1)\right\vert$}
In this section, we mainly discuss the long-time behavior of the solution of the Cauchy problem for the EB equation \eqref{1} in the solitonless
sector via the $\bar{\partial}$ descent method. We consider $y>0$ and the similarity region $\left\vert\frac{y}{t}\right\vert\leq N$, where $N$ is a
constant. In terms of $\frac{\mathrm{d}\theta}{\mathrm{d}\lambda}=0$, we can derive the two stationary points $\pm\lambda_{0}$, where
\begin{align}\label{34}
\lambda_{0}=\sqrt{\frac{y}{12t}},\quad \theta(\lambda)=4\lambda^{3}-12\lambda_{0}^{2}\lambda.
\end{align}
By calculating the real part of $2it\theta(\lambda), i.e., Re(2it\theta)=-24t\left((Re\lambda)^{2}-\frac{(Im\lambda)^{2}}{3}-\lambda_{0}^{2}\right)Im\lambda$, the exponential decaying domains of the oscillation term can be shown in Fig. \ref{F1}.\\
\centerline{\begin{tikzpicture}[scale=1.0]
\draw[-][thick](-3,0)--(-2,0);
\draw[-][thick](-2.0,0)--(-1.0,0)node[below]{$-\lambda_{0}$};
\draw[-][thick](-1,0)--(0,0);
\draw[-][thick](0,0)--(1,0)node[below]{$\lambda_{0}$};
\draw[-][thick](1,0)--(2,0);
\draw[fill] (1,0) circle [radius=0.035];
\draw[fill] (-1,0) circle [radius=0.035];
\draw[-][thick](2.0,0)--(3.0,0);
\draw [-,thick, cyan] (1,0) to [out=90,in=-120] (1.9,3);
\draw [-,thick, cyan] (1,0) to [out=-90,in=120] (1.9,-3);
\draw [-,thick, cyan] (-1,0) to [out=90,in=-60] (-1.9,3);
\draw [-,thick, cyan] (-1,0) to [out=-90,in=60] (-1.9,-3);
\draw[fill] (0,2) node{$e^{-2it\theta}\rightarrow 0$};
\draw[fill] (0,-2) node{$e^{2it\theta}\rightarrow 0$};
\draw[fill] (-2,1) node{$e^{2it\theta}\rightarrow 0$};
\draw[fill] (2,1) node{$e^{2it\theta}\rightarrow 0$};
\draw[fill] (2,-1) node{$e^{-2it\theta}\rightarrow 0$};
\draw[fill] (-2,-1) node{$e^{-2it\theta}\rightarrow 0$};
\end{tikzpicture}}\\
\begin{figure}[!htb]
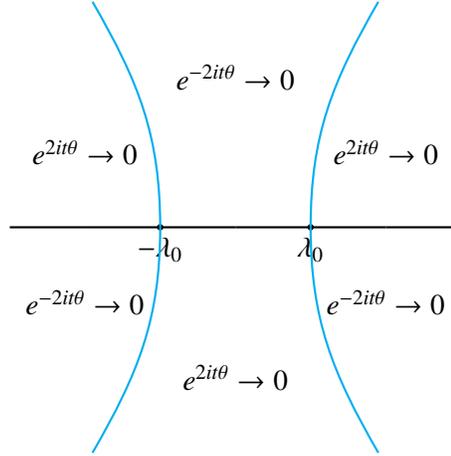

\centering
\caption{\footnotesize Exponential decaying domains.}
\label{F1}
\end{figure}

\subsection{The first deformation of the basic RH problem}
According to Fig. \ref{F1}, we need to decompose the jump
matrix $J(y, t, \lambda)$ into
\begin{align}\label{35}
J(y, t, \lambda)=
\begin{cases}
\left(\begin{array}{cc}
    1  &  0\\
    \overline{r(\lambda)}e^{2it\theta}  &  1\\
\end{array}\right)\left(\begin{array}{cc}
    1  &  r(\lambda)e^{-2it\theta}\\
    0  &  1\\
\end{array}\right),\ \left\vert\lambda\right\vert<\lambda_{0}, \\
\left(\begin{array}{cc}
    1  &  \frac{r(\lambda)}{1+\mid r(\lambda) \mid^{2}}e^{-2it\theta}\\
    0  &  1\\
\end{array}\right)\left(\begin{array}{cc}
    \frac{1}{1+\mid r(\lambda) \mid^{2}}  &  0\\
    0  &  1+\mid r(\lambda) \mid^{2}\\
\end{array}\right)\left(\begin{array}{cc}
    1  &  0\\
    \frac{\overline{r(\lambda)}}{1+\left\vert r(\lambda) \right\vert^{2}}e^{2it\theta}  &  1\\
\end{array}\right),\ \left\vert\lambda\right\vert>\lambda_{0},
\end{cases}
\end{align}
and  introduce a scalar function $\delta(\lambda)$, given by
\begin{align}\label{36}
\delta(\lambda)=\mbox{exp}\left[i\left(\int_{-\infty}^{-\lambda_{0}}+\int_{\lambda_{0}}^{+\infty}\right)\frac{\nu(s)}{s-\lambda}\mathrm{d}s\right],\
\nu(s)=-\frac{1}{2\pi}\ln(1+\left\vert r(s)\right\vert^{2}),
\end{align}
which satisfies the following Proposition.

\begin{prop}
 The function $\delta(\lambda)$ admits that:\\

 $\bullet$ $\delta(\lambda)$ is meromorphic in $(-\lambda_{0},\lambda_{0})$;\\

 $\bullet$ For $\lambda\in(-\infty,-\lambda_{0})\cup(\lambda_{0},+\infty)$, one has
\begin{align}\label{37} \delta_{+}(\lambda)=\delta_{-}(\lambda)(1+\left\vert r(\lambda) \right\vert^{2}),\ \lambda\in(-\infty,-\lambda_{0})\cup(\lambda_{0},+\infty);\end{align}

$\bullet$ As $\left\vert \lambda\right\vert\rightarrow \infty$ with $\left\vert arg(\lambda)\right\vert\leq c <\pi$, $\delta(\lambda)=1-\frac{i}{\lambda}\left(\int_{-\infty}^{-\lambda_{0}}+\int_{\lambda_{0}}^{+\infty}\right)\nu(s)\mathrm{d}s+O(\lambda^{-2})$;

$\bullet$ As $\lambda\rightarrow 0$, one has $\delta(\lambda)=1+\delta_{1}\lambda+O(\lambda^{2}),\delta_{1}=2i\int_{\lambda_{0}}^{+\infty}\frac{\nu(s)}{s^{2}}\mathrm{d}s$;

$\bullet$ As $\lambda\rightarrow \pm\lambda_{0}$ along ray $\lambda=\pm\lambda_{0}+e^{i\varphi}\mathbb{R}_{+}$ with $\left\vert \varphi\right\vert \leq c < \pi$, we have
\begin{align}\label{38} \left\vert\delta(\lambda)-\delta_{0}(\pm\lambda_{0})(\lambda\mp\lambda_{0})^{\mp i\nu(\pm\lambda_{0})}\right\vert\lesssim\left\vert\lambda\mp\lambda_{0}\right\vert^{\frac{1}{2}},\end{align}

where  \begin{align}\label{39} &\delta_{0}(\pm\lambda_{0})=\rm{exp}\{i\beta(\pm\lambda_{0},\pm\lambda_{0})\},\notag\\
&\beta(\lambda,\pm\lambda_{0})=\pm\nu(\pm\lambda_{0})\ln(\lambda\mp(\lambda_{0}+1))+\left(\int_{-\infty}^{-\lambda_{0}}
+\int_{\lambda_{0}}^{+\infty}\right)\frac{\nu(s)-\chi_{\pm}(s)\nu(\pm\lambda_{0})}{s-\lambda}\mathrm{d}s,\notag\\
&\chi_{+}(\lambda)=\begin{cases}1,\quad \lambda_{0}<\lambda<\lambda_{0}+1,\\
0,\quad \mbox{elsewhere}, \end{cases} \quad \chi_{-}(\lambda)=\begin{cases}1,\quad -\lambda_{0}-1<\lambda<-\lambda_{0},\\
0,\quad \mbox{elsewhere}. \end{cases}\end{align}
\end{prop}

Next, taking
\begin{align}\label{40} M^{(1)}(\lambda)=M(\lambda)\delta^{\sigma_{3}},\end{align}
we get a new matrix-value function $M^{(1)}(\lambda)$ which satisfies the following matrix RH problem.

\begin{rhp}\label{rhp3}
Find an analysis function $M^{(1)}(y, t; \lambda)$ which satisfies:\\

$\bullet$ $M^{(1)}(y, t; \lambda)$ is meromorphic in $\mathbb{C}\setminus \mathbb{R}$;\\

$\bullet$ $M_{+}^{(1)}(y, t; \lambda)=M^{(1)}_{-}(y, t; \lambda)J^{(1)}(y, t; \lambda), \lambda\in \mathbb{R}$, where
\begin{align}\label{41}
J^{(1)}(y, t; \lambda)=\begin{cases}
\left(\begin{array}{cc}
    1  &  0\\
    \overline{r(\lambda)}\delta^{2}e^{2it\theta}  &  1\\
\end{array}\right)\left(\begin{array}{cc}
    1  &  r(\lambda)\delta^{-2}e^{-2it\theta}\\
    0  &  1\\
\end{array}\right),\ \left\vert\lambda\right\vert<\lambda_{0}, \\
\left(\begin{array}{cc}
    1  &  \frac{r(\lambda)}{1+\mid r(\lambda) \mid^{2}}\delta_{-}^{-2}e^{-2it\theta}\\
    0  &  1\\
\end{array}\right)\left(\begin{array}{cc}
    1  &  0\\
    \frac{\overline{r(\lambda)}}{1+\mid r(\lambda) \mid^{2}}\delta_{+}^{2}e^{2it\theta}  &  1\\
\end{array}\right),\ \left\vert\lambda\right\vert>\lambda_{0};
\end{cases}
\end{align}

$\bullet$ $M^{(1)}(y, t; \lambda)=\mathbb{I}+O(\frac{1}{\lambda})$ as $\lambda\rightarrow\infty$.
\end{rhp}

Writing the $M^{(1)}(y,t; \lambda), M(y,t; \lambda)$ into the following progressive expansion form
\begin{align}\label{42} &M^{(1)}(y,t; \lambda)=M_{0}^{(1)}(y,t)+M_{1}^{(1)}(y,t)\lambda+O(\lambda^{2}),\quad \lambda\rightarrow 0,\notag\\
&M(y,t; \lambda)=M_{0}(y,t)+M_{1}(y,t)\lambda+O(\lambda^{2}),\quad \lambda\rightarrow 0,\end{align}
and combining the transformation \eqref{40}, we easily derive
\begin{align}\label{44} M_{0}(y,t)=M^{(1)}_{0}(y,t),\quad M_{1}(y,t)=M^{(1)}_{1}(y,t)-\delta_{1}M^{(1)}_{0}(y,t)\sigma_{3}. \end{align}
Hence, we have
\begin{align}\label{45} &iq(y,t)=\left[\left(M^{(1)}_{0}(y,t)\right)^{-1}M^{(1)}_{1}(y,t)\right]_{12},\notag\\
&ic_{+}=\left[\left(M^{(1)}_{0}(y,t)\right)^{-1}M^{(1)}_{1}(y,t)\right]_{11}-\delta_{1}. \end{align}

\subsection{The construction of the mixed $\bar{\partial}$-RH problem}
In order to construct the mixed $\bar{\partial}$-RH problem, we perform the continuous extensions of the jump
matrix off the real axis to deform  the oscillatory jump into new contours along which the jumps are decaying. We firstly define the contours
\begin{align}\label{46} &\Sigma_{1}^{\pm}=\pm\lambda_{0}+e^{\frac{(2\mp 1)i\pi}{4}}\mathbb{R}_{+}, \quad
\Sigma_{2}^{\pm}=\pm\lambda_{0}+e^{\frac{(-2\pm 1)i\pi}{4}}\mathbb{R}_{+},\notag\\
&\Sigma_{3}^{\pm}=\pm\lambda_{0}+e^{\frac{(2\pm 1)i\pi}{4}}h, \ \Sigma_{4}^{\pm}=\pm\lambda_{0}+e^{\frac{(-2\mp 1)i\pi}{4}}h, \ h\in(0,\sqrt{2}\lambda_{0}), \end{align}
which separate the complex plane $\mathbb{C}$ into ten open sectors $\Omega_{j}^{\pm}(j=1,\cdots, 4),\Omega_{5} $ and $\Omega_{6}$, see Fig. \ref{F2}.\\
\centerline{\begin{tikzpicture}[scale=2.0]
\draw[->][thick](-3,0)--(-2,0);
\draw[-][thick](-2.0,0)--(-1.0,0)node[below]{$-\lambda_{0}$};
\draw[-][thick](-1,0)--(0,0);
\draw[->][thick](0,1)--(0.5,0.5);
\draw[-][thick](0.5,0.5)--(1,0);
\draw[->][thick](0,-1)--(0.5,-0.5);
\draw[-][thick](0.5,-0.5)--(1,0);
\draw[-][thick](2,-1)--(1.5,-0.5);
\draw[<-][thick](1.5,-0.5)--(1,0);
\draw[-][thick](2,1)--(1.5,0.5);
\draw[<-][thick](1.5,0.5)--(1,0);
\draw[-][thick](0,1)--(-0.5,0.5);
\draw[->][thick](-1,0)--(-0.5,0.5);
\draw[-][thick](0,-1)--(-0.5,-0.5);
\draw[->][thick](-1,0)--(-0.5,-0.5);
\draw[-][thick](-1,0)--(-1.5,0.5);
\draw[<-][thick](-1.5,0.5)--(-2,1);
\draw[-][thick](-1,0)--(-1.5,-0.5);
\draw[<-][thick](-1.5,-0.5)--(-2,-1);
\draw[-][thick](0,0)--(1,0)node[below]{$\lambda_{0}$};
\draw[->][thick](1,0)--(2,0);
\draw[fill] (1,0) circle [radius=0.035];
\draw[fill] (-1,0) circle [radius=0.035];
\draw[-][thick](2.0,0)--(3.0,0);
\draw[fill] (0.3,1.2) node{$\Omega_{5}$};
\draw[fill] (0.3,-1.2) node{$\Omega_{6}$};
\draw[fill] (0.5,0.25) node{$\Omega_{3}^{+}$};
\draw[fill] (0.5,-0.25) node{$\Omega_{4}^{+}$};
\draw[fill] (-0.5,0.25) node{$\Omega_{3}^{-}$};
\draw[fill] (-0.5,-0.25) node{$\Omega_{4}^{-}$};
\draw[fill] (2,-0.5) node{$\Omega_{2}^{+}$};
\draw[fill] (2,0.5) node{$\Omega_{1}^{+}$};
\draw[fill] (-2,-0.5) node{$\Omega_{2}^{-}$};
\draw[fill] (-2,0.5) node{$\Omega_{1}^{-}$};
\draw[->][thick](0,1.5)--(0,0.5);
\draw[->][thick](0,0.5)--(0,-0.5);
\draw[-][thick](0,-0.5)--(0,-1.5);
\draw[fill] (0.12,0.4) node{$\Sigma_{5}$};
\draw[fill] (0.12,-0.4) node{$\Sigma_{5}$};
\draw[fill] (-1.8,1) node{$\Sigma_{1}^{-}$};
\draw[fill] (-1.8,-1) node{$\Sigma_{2}^{-}$};
\draw[fill] (1.8,1) node{$\Sigma_{1}^{+}$};
\draw[fill] (1.8,-1) node{$\Sigma_{2}^{+}$};
\draw[fill] (-0.5,0.7) node{$\Sigma_{3}^{-}$};
\draw[fill] (-0.5,-0.7) node{$\Sigma_{4}^{-}$};
\draw[fill] (0.5,0.7) node{$\Sigma_{3}^{+}$};
\draw[fill] (0.5,-0.7) node{$\Sigma_{4}^{+}$};
\end{tikzpicture}}\\
\begin{figure}[!htb]
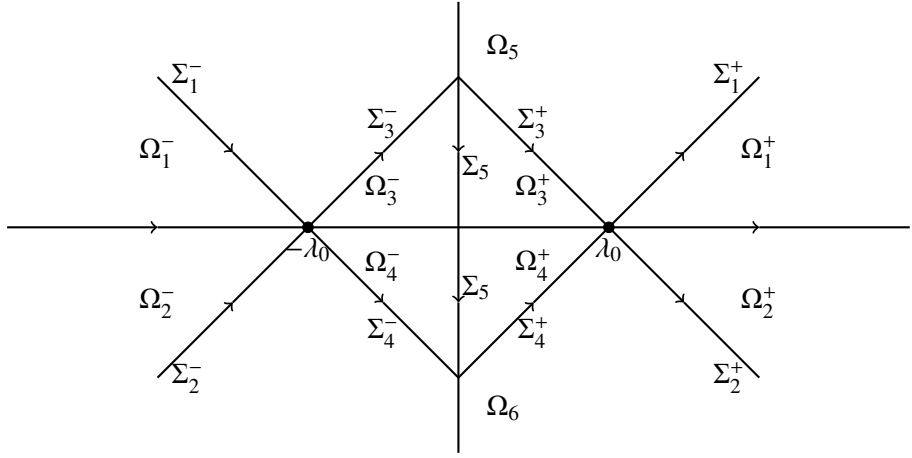

\centering
\caption{\footnotesize Deformation from $\mathbb{R}$ to the new contour $\Sigma=\Sigma_{1}^{\pm}\cup\Sigma_{2}^{\pm}\cup\Sigma_{3}^{\pm}\cup\Sigma_{4}^{\pm}\cup\Sigma_{5}, \Sigma_{5}=(-i\lambda_{0},-i\lambda_{0}\tan(\frac{\pi}{12}))\cup(-i\lambda_{0}\tan(\frac{\pi}{12}),
i\lambda_{0}\tan(\frac{\pi}{12}))\cup(i\lambda_{0}\tan(\frac{\pi}{12}),i\lambda_{0})$.}
\label{F2}
\end{figure}

\begin{prop}\label{prop1}
There are functions $R_{j}\rightarrow \mathbb{C}, j=1,\cdots, 4$ such that
\begin{align}\label{47} &R_{1}^{\pm}=\begin{cases}
r(\lambda)\delta^{-2}(\lambda) \qquad \lambda\in I_{\pm},\ I_{+}=(\lambda_{0},+\infty),\ I_{-}=(-\infty,-\lambda_{0}),\\
r(\pm\lambda_{0})\delta_{0}^{-2}(\pm\lambda_{0})(\lambda\mp\lambda_{0})^{-2i\nu(\pm\lambda_{0})} \qquad \lambda\in \Sigma_{3}^{\pm},\\
\end{cases}\notag\\
&R_{2}^{\pm}=\begin{cases}
\overline{r(\lambda)}\delta^{2}(\lambda) \qquad \lambda\in I_{\pm},\\
\overline{r(\pm\lambda_{0})}\delta_{0}^{2}(\pm\lambda_{0})(\lambda\mp\lambda_{0})^{2i\nu(\pm\lambda_{0})} \qquad \lambda\in \Sigma_{4}^{\pm},\\
\end{cases}\notag\\
&R_{3}^{\pm}=\begin{cases}
\frac{\overline{r(\lambda)}}{1+\mid r(\lambda)\mid^{2}}\delta_{+}^{2}(\lambda) \qquad \lambda\in (-\lambda_{0},\lambda_{0}),\\
\frac{\overline{r(\pm\lambda_{0})}}{1+\mid r(\pm\lambda_{0})\mid^{2}}\delta_{0}^{2}(\pm\lambda_{0})(\lambda\mp\lambda_{0})^{2i\nu(\pm\lambda_{0})} \qquad \lambda\in \Sigma_{1}^{\pm},\\
\end{cases}\notag\\
&R_{4}^{\pm}=\begin{cases}
\frac{r(\lambda)}{1+\mid r(\lambda)\mid^{2}}\delta_{-}^{-2}(\lambda) \qquad \lambda\in (-\lambda_{0},\lambda_{0}),\\
\frac{r(\pm\lambda_{0})}{1+\mid r(\pm\lambda_{0})\mid^{2}}\delta_{0}^{-2}(\pm\lambda_{0})(\lambda\mp\lambda_{0})^{-2i\nu(\pm\lambda_{0})} \qquad \lambda\in \Sigma_{2}^{\pm},\\
\end{cases}
\end{align}
which admit the following estimates
\begin{align}\label{48} &\left\vert R_{j}^{\pm}\right\vert\lesssim \sin^{2}\left(arg(\lambda\mp\lambda_{0})\right)+<Re\lambda>^{-1/2},\notag\\
&\left\vert\bar{\partial} R_{j}^{\pm}\right\vert\lesssim \left\vert\lambda\mp\lambda_{0}\right\vert^{-1/2}+\left\vert(p_{j}^{\pm})'(Re\lambda)\right\vert,\notag\\
&\left\vert\bar{\partial} R_{j}^{\pm}\right\vert=0, \qquad \mbox{if} \qquad \lambda\in \Omega_{5}\cup\Omega_{6}, \end{align}
where $<Re\lambda>=\sqrt{1+(Re\lambda)^{2}}, p_{1}^{\pm}=r(\lambda), p_{2}^{\pm}=\overline{r(\lambda)}, p_{3}^{\pm}=\frac{\overline{r(\lambda)}}{1+\mid r(\lambda)\mid^{2}}, p_{4}^{\pm}=\frac{r(\lambda)}{1+\mid r(\lambda)\mid^{2}}$.
\end{prop}

Taking a transformation
\begin{align}\label{49}
M^{(2)}=M^{(1)}R^{(2)},
\end{align}
where
\begin{align}\label{50}
R^{(2)}=\begin{cases}
\left(\begin{array}{cc}
    1  &  -R_{1}^{\pm}e^{-2it\theta}\\
    0  &  1\\
\end{array}\right),\qquad \lambda\in\Omega_{3}^{\pm},\\
\left(\begin{array}{cc}
    1  &  0\\
    R_{2}^{\pm}e^{2it\theta}  &  1\\
\end{array}\right),\qquad \lambda\in\Omega_{4}^{\pm},\\
\left(\begin{array}{cc}
    1  &  0\\
    -R_{3}^{\pm}e^{2it\theta}  &  1\\
\end{array}\right),\qquad \lambda\in\Omega_{1}^{\pm},\\
\left(\begin{array}{cc}
    1  &  R_{4}^{\pm}e^{-2it\theta}\\
    0  &  1\\
\end{array}\right),\qquad \lambda\in\Omega_{2}^{\pm},\\
\left(\begin{array}{cc}
    1  &  0\\
    0  &  1\\
\end{array}\right),\qquad \lambda\in\Omega_{5}\cup\Omega_{6},
\end{cases}
\end{align}
we can successfully construct the following $\bar{\partial}$-RH problem.

\begin{rhp}\label{rhp4}
 Find an analysis function $M^{(2)}(y, t; \lambda)$ which satisfies:\\

$\bullet$ $M^{(2)}(y, t; \lambda)$ is meromorphic in $\mathbb{C}\setminus \Sigma$;\\

$\bullet$ $M_{+}^{(2)}(y, t; \lambda)=M^{(2)}_{-}(y, t; \lambda)J^{(2)}(y, t; \lambda), \lambda\in \Sigma$, where
\begin{align}\label{51}
J^{(2)}(y, t; \lambda)=\begin{cases}
\left(\begin{array}{cc}
    1  &  R_{1}^{\pm}e^{-2it\theta}\\
    0  &  1\\
\end{array}\right),\qquad \lambda\in\Sigma_{3}^{\pm},\\
\left(\begin{array}{cc}
    1  &  0\\
    R_{2}^{\pm}e^{2it\theta}  &  1\\
\end{array}\right),\qquad \lambda\in\Sigma_{4}^{\pm},\\
\left(\begin{array}{cc}
    1  &  0\\
    R_{3}^{\pm}e^{2it\theta}  &  1\\
\end{array}\right),\qquad \lambda\in\Sigma_{1}^{\pm},\\
\left(\begin{array}{cc}
    1  &  R_{4}^{\pm}e^{-2it\theta}\\
    0  &  1\\
\end{array}\right),\qquad \lambda\in\Sigma_{2}^{\pm},\\
\left(\begin{array}{cc}
    1  &  (R_{1}^{-}-R_{1}^{+})e^{-2it\theta}\\
    0  &  1\\
\end{array}\right),\qquad \lambda\in (i\lambda_{0}\tan(\frac{\pi}{12}),i\lambda_{0}),\\
\left(\begin{array}{cc}
    1  &  0\\
    (R_{2}^{+}-R_{2}^{-})e^{2it\theta}  &  1\\
\end{array}\right),\qquad \lambda\in (-i\lambda_{0},-i\lambda_{0}\tan(\frac{\pi}{12})),\\
\left(\begin{array}{cc}
    1  &  0\\
    0  &  1\\
\end{array}\right),\qquad \lambda\in (-i\lambda_{0}\tan(\frac{\pi}{12}),i\lambda_{0}\tan(\frac{\pi}{12})).
\end{cases}
\end{align}

$\bullet$ $M^{(2)}(y, t; \lambda)=\mathbb{I}+O(\frac{1}{\lambda})$, as $\lambda\rightarrow\infty$.

$\bullet$ $\bar{\partial}M^{(2)}=M^{(2)}\bar{\partial}R^{(2)}(\lambda)$, as $\lambda\in\mathbb{C}\setminus \Sigma$, where
\begin{align}\label{52}
\bar{\partial}R^{(2)}=\begin{cases}
\left(\begin{array}{cc}
    0  &  -\bar{\partial}R_{1}^{\pm}e^{-2it\theta}\\
    0  &  0\\
\end{array}\right),\qquad \lambda\in\Omega_{3}^{\pm},\\
\left(\begin{array}{cc}
    0  &  0\\
    \bar{\partial}R_{2}^{\pm}e^{2it\theta}  &  0\\
\end{array}\right),\qquad \lambda\in\Omega_{4}^{\pm},\\
\left(\begin{array}{cc}
    0  &  0\\
    -\bar{\partial}R_{3}^{\pm}e^{2it\theta}  &  0\\
\end{array}\right),\qquad \lambda\in\Omega_{1}^{\pm},\\
\left(\begin{array}{cc}
    0  &  \bar{\partial}R_{4}^{\pm}e^{-2it\theta}\\
    0  &  0\\
\end{array}\right),\qquad \lambda\in\Omega_{2}^{\pm},\\
\left(\begin{array}{cc}
    0  &  0\\
    0  &  0\\
\end{array}\right),\qquad \lambda\in\Omega_{5}\cup\Omega_{6}.
\end{cases}
\end{align}
\end{rhp}

\subsection{The decomposition of the mixed $\bar{\partial}$-RH problem}
To solve the mixed $\bar{\partial}$-RH problem, i.e., RH  problem \ref{rhp4}, we decompose it into  a pure RH problem with $\bar{\partial}R^{(2)}=0$ and a pure $\bar{\partial}$-RH problem with $\bar{\partial}R^{(2)}\neq0$.

\subsubsection{Pure RH problem}
To solve the pure RH problem, we need to obtain a model RH problem. Firstly, we construct a RH problem for $M^{RHP}(y, t; \lambda)$.

\begin{rhp}\label{rhp5}
Find an analysis function $M^{RHP}(y, t; \lambda)$ which satisfies:\\

$\bullet$ $M^{RHP}(y, t; \lambda)$ is meromorphic in $\mathbb{C}\setminus \Sigma$;\\

$\bullet$ $M_{+}^{RHP}(y, t; \lambda)=M^{RHP}_{-}(y, t; \lambda)J^{(2)}(y, t; \lambda), \lambda\in \Sigma$, where $J^{(2)}(y, t; \lambda)$ is given in \eqref{51}.

$\bullet$ $M^{RHP}(y, t; \lambda)=\mathbb{I}+O(\frac{1}{\lambda})$, as $\lambda\rightarrow\infty$.
\end{rhp}
Next, we commit  to solve the RH  problem \ref{rhp5} by decomposing $M^{RHP}(y, t; \lambda)$ into two parts
\begin{align}\label{54}
M^{RHP}(\lambda)=\begin{cases}
E(\lambda),\qquad \lambda\in \mathbb{C}\backslash (\mathcal{U}_{\lambda_{0}}\cup\mathcal{U}_{-\lambda_{0}}),\\
E(\lambda)M^{(\lambda_{0})},\qquad \lambda\in \mathcal{U}_{\lambda_{0}},\\
E(\lambda)M^{(-\lambda_{0})},\qquad \lambda\in \mathcal{U}_{-\lambda_{0}},\\
\end{cases}
 \end{align} where $\mathcal{U}_{\pm\lambda_{0}}$ are the open neighborhood, given by
\begin{align}\label{53}
\mathcal{U}_{\pm\lambda_{0}}=\{\lambda:\left\vert \lambda\mp \lambda_{0}\right\vert\leq \epsilon\},
 \end{align}
of which $\epsilon$ is  an arbitrary  small constant.  $M^{(\pm\lambda_{0})}$ can be solved by the parabolic cylinder model,
and $E(\lambda)$ is the solution of a small norm RH problem.
\begin{prop}\label{prop2}
The jump matrix $J^{(2)}$ given in \eqref{51} meets the following estimates
\begin{align}\label{55}
\left\Vert J^{(2)}(\lambda)-\mathbb{I}\right\Vert_{L^{\infty}(\Sigma)}=\begin{cases}
O(e^{-24t\lambda_{0}\epsilon^{2}}),\qquad \lambda\in \Sigma_{j}^{\pm}\backslash (\mathcal{U}_{\lambda_{0}}\cup\mathcal{U}_{-\lambda_{0}}),\ j=1, 2,\\
O(e^{-8t\lambda_{0}\epsilon^{2}}),\qquad \lambda\in \Sigma_{j}^{\pm}\backslash (\mathcal{U}_{\lambda_{0}}\cup\mathcal{U}_{-\lambda_{0}}),\ j=3, 4,\\
O(t^{-\frac{1}{2}}\lambda_{0}^{-1}\left\vert \lambda\mp \lambda_{0}\right\vert^{-1}),\qquad \lambda\in \Sigma\cap\mathcal{U}_{\pm\lambda_{0}},\\
O(e^{-8t\lambda_{0}^{3}\tan^{3}(\frac{\pi}{12})}),\qquad \lambda\in (\pm i\lambda_{0}\tan(\frac{\pi}{12}),\pm i\lambda_{0}),\\
0,\qquad \lambda\in (- i\lambda_{0}\tan(\frac{\pi}{12}),\pm i\lambda_{0}\tan(\frac{\pi}{12})).\\
\end{cases}
 \end{align}
\end{prop}
\begin{proof}
We mainly prove the case of $\lambda\in \Sigma_{1}^{+}$, other cases can be presented in a similar way. On $\Sigma_{1}^{+}$, the  jump contour is $\lambda-\lambda_{0}=\left\vert\lambda-\lambda_{0}\right\vert e^{\frac{i\pi}{4}}$, then we get
\begin{align}\label{56}
Im(\theta)=2\left\vert \lambda-\lambda_{0}\right\vert^{2}\left(\sqrt{2}
\left\vert \lambda-\lambda_{0}\right\vert +6\lambda_{0}\right).
 \end{align}
Using $\eqref{48}, \eqref{51}, \eqref{56}$, we obtain
\begin{align}\label{57}
\left\vert R_{3}^{+}e^{2it\theta}\right\vert \leq \left\vert R_{3}^{+}\right\vert \left\vert e^{-2tIm(\theta)}\right\vert \lesssim e^{-4t\left\vert \lambda-\lambda_{0}\right\vert^{2}\left(\sqrt{2}
\left\vert \lambda-\lambda_{0}\right\vert +6\lambda_{0}\right)}\lesssim e^{-24t\lambda_{0}\left\vert \lambda-\lambda_{0}\right\vert^{2}}.
 \end{align}
Therefore, for $\lambda\in \Sigma_{1}^{+}\cap \mathcal{U}_{\lambda_{0}}$, we arrive at
\begin{align}\label{58}
\left\Vert J^{(2)}-\mathbb{I}\right\Vert_{L^{\infty}(\Sigma)}\lesssim t^{-\frac{1}{2}}\lambda_{0}^{-1}\left\vert \lambda-\lambda_{0}\right\vert^{-1}.
 \end{align}
It can be seen that in $\mathcal{U}_{\lambda_{0}}$, the jump matrix is monotonically nonuniformly decaying point by point to the identity matrix.

For $\lambda \in \Sigma_{1}^{+}\cap \{\left\vert \lambda-\lambda_{0}\right\vert\geq \epsilon\}$, we have
\begin{align}\label{59}
\left\Vert J^{(2)}-\mathbb{I}\right\Vert_{L^{\infty}(\Sigma)}\lesssim e^{-24t\lambda_{0}\epsilon^{2}}.
 \end{align}
\end{proof}

According to \eqref{55}, for large time, the $J^{(2)}(\lambda)-\mathbb{I}$ does not have a uniform estimate when $\lambda\in \mathcal{U}_{\pm\lambda_{0}}$, which implies that we can establish a local solvable model for the error function $E(\lambda)$
by introducing a model $M^{(\pm\lambda_{0})}$ to match the jumps of $M^{RHP}$ on
$\Sigma\cap (\mathcal{U}_{\lambda_{0}}\cup\mathcal{U}_{-\lambda_{0}})$.

Recalling the definition of $\theta(\lambda)$, we can introduce the transformation
\begin{align}\label{60}
\xi=\xi(\lambda)=\sqrt{48\lambda_{0}t}(\lambda\pm\lambda_{0})
 \end{align}
and the scaling operators
\begin{align}\label{61}
&f(\lambda)\mapsto (N_{A}f)(\lambda)=f(\frac{\xi}{\sqrt{48\lambda_{0}t}}-\lambda_{0}),\notag\\
&f(\lambda)\mapsto (N_{B}f)(\lambda)=f(\frac{\xi}{\sqrt{48\lambda_{0}t}}+\lambda_{0}).
 \end{align}
Therefore, considering the scaling operator  $N_{B}$, we have
\begin{align}\label{62}
&N_{B}(R_{1}^{+}e^{-2it\theta})=r_{0}\xi^{-2i\nu(\lambda_{0})}e^{-\frac{i\xi^{3}}{6\lambda_{0}\sqrt{48\lambda_{0}t}}-\frac{i\xi^{2}}{2}},\notag\\
&N_{B}(R_{2}^{+}e^{2it\theta})=\bar{r}_{0}\xi^{2i\nu(\lambda_{0})}e^{\frac{i\xi^{3}}{6\lambda_{0}\sqrt{48\lambda_{0}t}}+\frac{i\xi^{2}}{2}},\notag\\
&N_{B}(R_{3}^{+}e^{2it\theta})=\frac{\bar{r}_{0}}{1+\mid r_{0}\mid^{2}}\xi^{2i\nu(\lambda_{0})}e^{\frac{i\xi^{3}}{6\lambda_{0}\sqrt{48\lambda_{0}t}}+\frac{i\xi^{2}}{2}},\notag\\
&N_{B}(R_{4}^{+}e^{-2it\theta})=\frac{r_{0}}{1+\mid r_{0}\mid^{2}}\xi^{-2i\nu(\lambda_{0})}e^{-\frac{i\xi^{3}}{6\lambda_{0}\sqrt{48\lambda_{0}t}}-\frac{i\xi^{2}}{2}},
 \end{align}
where we have taken $r_{0}=r(\lambda_{0})\delta_{0}^{-2}(\lambda_{0})e^{16it\lambda_{0}^{3}}e^{i\nu(\lambda_{0})\ln(48t\lambda_{0})}$.
Then, we can get the following RH problem $M^{(\lambda_{0})}$ as $t\rightarrow\infty$ in
the $\xi$ plane, see Fig. \ref{F3}.\\
\centerline{\begin{tikzpicture}
\draw[fill] (0,0) circle [radius=0.035];
\draw[->][thick](-2.12,-2.12)--(-1.06,-1.06);
\draw[-][thick](-1.06,-1.06)--(0,0)node[below]{$O$};
\draw[->][thick](0,0)--(1.06,1.06);
\draw[-][thick](1.06,1.06)--(2.12,2.12);
\draw[->][thick](-2.12,2.12)--(-1.06,1.06);
\draw[-][thick](-1.06,1.06)--(0,0);
\draw[->][thick](0,0)--(1.06,-1.06);
\draw[->][dashed](-3,0)--(-1.5,0);
\draw[->][dashed](-1.5,0)--(1.5,0);
\draw[-][dashed](1.5,0)--(3.0,0)node[right]{$\mathbb{R}$};
\draw[-][thick](1.06,-1.06)--(2.12,-2.12);
\draw[fill] (0,1) node{$\Omega_{0}$};
\draw[fill] (0,-1) node{$\Omega_{0}$};
\draw[fill] (-1,0.5) node{$\Omega_{3}$};
\draw[fill] (1,0.5) node{$\Omega_{1}$};
\draw[fill] (-1,-0.5) node{$\Omega_{4}$};
\draw[fill] (1,-0.5) node{$\Omega_{2}$};
\draw[fill] (-1.6,2.0) node{$\widetilde{\Sigma_{3}}$};
\draw[fill] (-1.6,-2.0) node{$\widetilde{\Sigma_{4}}$};
\draw[fill] (1.6,2.0) node{$\widetilde{\Sigma_{1}}$};
\draw[fill] (1.6,-2.0) node{$\widetilde{\Sigma_{2}}$};
\end{tikzpicture}}\\
\begin{figure}[!htb]
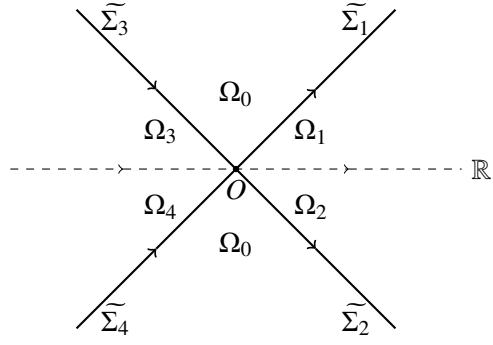

\centering
\caption{\footnotesize The jump contour $\widetilde{\Sigma_{B}}=\widetilde{\Sigma_{1}}\cup\widetilde{\Sigma_{2}}\cup\widetilde{\Sigma_{3}}\cup\widetilde{\Sigma_{4}}$  and domains $\Omega_{j}(j=0,\cdots, 4)$.}
\label{F3}
\end{figure}

\begin{rhp}\label{rhp6}
The analysis function $M^{(\lambda_{0})}(\xi)$ has the
following properties:\\

 $\bullet$ $M^{(\lambda_{0})}(\xi)$ is meromorphic in $\mathbb{C}\setminus \widetilde{\Sigma_{B}}$;\\

 $\bullet$ $M_{+}^{(\lambda_{0})}(\xi)=M^{(\lambda_{0})}_{-}(\xi)J^{(\lambda_{0})}(\xi), \lambda\in \widetilde{\Sigma_{B}}$, where
\begin{align}\label{63}
J^{(\lambda_{0})}(\xi)=
\begin{cases}
\left(\begin{array}{cc}
    1  &  r_{0}\xi^{-2i\nu(\lambda_{0})}e^{-\frac{i\xi^{2}}{2}}\\
    0  &  1\\
\end{array}\right),\qquad \xi\in\widetilde{\Sigma_{3}},\\
\left(\begin{array}{cc}
    1  &  0\\
    \bar{r}_{0}\xi^{2i\nu(\lambda_{0})}e^{\frac{i\xi^{2}}{2}}  &  1\\
\end{array}\right),\qquad \xi\in\widetilde{\Sigma_{4}},\\
\left(\begin{array}{cc}
    1  &  0\\
    \frac{\bar{r}_{0}}{1+\mid r_{0}\mid^{2}}\xi^{2i\nu(\lambda_{0})}e^{\frac{i\xi^{2}}{2}}  &  1\\
\end{array}\right),\qquad \xi\in\widetilde{\Sigma_{1}},\\
\left(\begin{array}{cc}
    1  &  \frac{r_{0}}{1+\mid r_{0}\mid^{2}}\xi^{-2i\nu(\lambda_{0})}e^{-\frac{i\xi^{2}}{2}}\\
    0  &  1\\
\end{array}\right),\qquad \xi\in\widetilde{\Sigma_{2}};
\end{cases}
\end{align}

$\bullet$ $M^{(\lambda_{0})}(\xi)=\mathbb{I}+O(\frac{1}{\xi})$ as $\xi\rightarrow\infty$.
\end{rhp}
As shown in ``Appendix \ref{A}", the RH problem \ref{rhp6} can be solved explicitly via using the parabolic
cylinder model problem.  For the another parabolic cylinder model $M^{(-\lambda_{0})}(\xi)$, it can be derived by the symmetric relation
\begin{align}\label{64.1}
M^{(-\lambda_{0})}(\xi)=\overline{M^{(\lambda_{0})}(-\bar{\xi})},\quad M_{1}^{(-\lambda_{0})}=-\overline{M_{1}^{(\lambda_{0})}}.
\end{align}

Since the jump matrix of  $M^{(\pm\lambda_{0})}$ and $M^{RHP}$ is consistent in disk $\mathcal{U}_{\pm\lambda_{0}}$, the matrix $E(\lambda)$ erects the jump of $M^{RHP}$ inside disk $\mathcal{U}_{\pm\lambda_{0}}$, and there is still a jump from $M^{RHP}$ outside the disk, so the jump path of $E(\lambda)$ is
\begin{align}\label{64}
\Sigma^{E}=\partial\mathcal{U}_{-\lambda_{0}}\cup\partial\mathcal{U}_{\lambda_{0}}\cup
\left(\Sigma\setminus\left(\mathcal{U}_{\lambda_{0}}\cup\mathcal{U}_{-\lambda_{0}}\right)\right)
 \end{align}
with clockwise direction for $\partial\mathcal{U}_{\pm\lambda_{0}}$. Then, the $E(\lambda)$ solves the following
RH problem, see Fig.\ref{F4}.\\
\centerline{\begin{tikzpicture}[scale=2.0]
\draw[->][thick](0,1)--(0,0.5);
\draw[-][thick](0,0.5)--(0,0);
\draw[->][thick](0,0)--(0,-0.5);
\draw[-][thick](0,-0.5)--(0,-1);
\draw[->][thick](0,1)--(0.5,0.5);
\draw[-][thick](0.5,0.5)--(0.65,0.35);
\draw[->][thick](0,-1)--(0.5,-0.5);
\draw[-][thick](0.5,-0.5)--(0.65,-0.35);
\draw[-][thick](2,-1)--(1.5,-0.5);
\draw[<-][thick](1.5,-0.5)--(1.35,-0.35);
\draw[-][thick](2,1)--(1.5,0.5);
\draw[<-][thick](1.5,0.5)--(1.35,0.35);
\draw[-][thick](0,1)--(-0.5,0.5);
\draw[->][thick](-0.65,0.35)--(-0.5,0.5);
\draw[-][thick](0,-1)--(-0.5,-0.5);
\draw[->][thick](-0.65,-0.35)--(-0.5,-0.5);
\draw[-][thick](-1.35,0.35)--(-1.5,0.5);
\draw[<-][thick](-1.5,0.5)--(-2,1);
\draw[-][thick](-1.35,-0.35)--(-1.5,-0.5);
\draw[<-][thick](-1.5,-0.5)--(-2,-1);
\draw[<-][thick](-0.5,0) arc(0:360:0.5);
\draw[-][thick](-0.5,0) arc(0:30:0.5);
\draw[-][thick](-0.5,0) arc(0:150:0.5);
\draw[-][thick](-0.5,0) arc(0:210:0.5);
\draw[-][thick](-0.5,0) arc(0:330:0.5);
\draw[<-][thick](1.5,0) arc(0:360:0.5);
\draw[-][thick](1.5,0) arc(0:30:0.5);
\draw[-][thick](1.5,0) arc(0:150:0.5);
\draw[-][thick](1.5,0) arc(0:210:0.5);
\draw[-][thick](1.5,0) arc(0:330:0.5);
\draw[fill] (-1,-0.62) node{$\partial\mathcal{U}_{-\lambda_{0}}$};
\draw[fill] (1,-0.62) node{$\partial\mathcal{U}_{\lambda_{0}}$};
\end{tikzpicture}}\\
\begin{figure}[!htb]
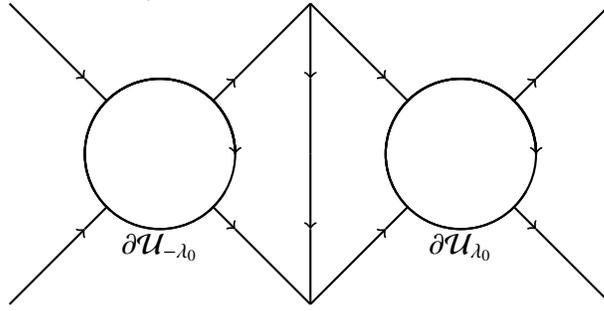

\centering
\caption{\footnotesize The jump contour $\Sigma^{E}$.}
\label{F4}
\end{figure}

\begin{rhp}\label{rhp7}
Find a matrix-valued function $E(\lambda)$ which satisfies:\\

 $\bullet$ $E(\lambda)$ is meromorphic in $\mathbb{C}\setminus \Sigma^{E}$;\\

 $\bullet$ $E_{+}(\lambda)=E_{-}(\lambda)J^{E}(\lambda), \lambda\in \Sigma^{E}$, where
\begin{align}\label{65}
J^{E}(\lambda)=
\begin{cases}
J^{(2)},\qquad \lambda\in\Sigma\setminus(\mathcal{U}_{-\lambda_{0}}\cup\mathcal{U}_{\lambda_{0}}),\\
M^{(\lambda_{0})},\qquad \lambda\in \partial\mathcal{U}_{\lambda_{0}},\\
M^{(-\lambda_{0})},\qquad \lambda\in \partial\mathcal{U}_{-\lambda_{0}}.
\end{cases}
\end{align}

$\bullet$ $E(\lambda)=\mathbb{I}+O(\frac{1}{\lambda})$ as $\lambda\rightarrow\infty$.
\end{rhp}

\begin{prop}\label{prop3}
The jump matrix $J^{E}$ defined in \eqref{65} arrives at the following estimates
\begin{align}\label{66}
\left\vert J^{E}(\lambda)-\mathbb{I}\right\vert=\begin{cases}
O(e^{-24t\lambda_{0}\epsilon^{2}}),\qquad \lambda\in \Sigma_{j}^{\pm}\backslash (\mathcal{U}_{\lambda_{0}}\cup\mathcal{U}_{-\lambda_{0}}),\ j=1, 2,\\
O(e^{-8t\lambda_{0}\epsilon^{2}}),\qquad \lambda\in \Sigma_{j}^{\pm}\backslash (\mathcal{U}_{\lambda_{0}}\cup\mathcal{U}_{-\lambda_{0}}),\ j=3, 4,\\
O(t^{-\frac{1}{2}}),\qquad \lambda\in \partial\mathcal{U}_{-\lambda_{0}}\cup\partial\mathcal{U}_{\lambda_{0}},\\
O(e^{-8t\lambda_{0}^{3}\tan^{3}(\frac{\pi}{12})}),\qquad \lambda\in (\pm i\lambda_{0}\tan(\frac{\pi}{12}),\pm i\lambda_{0}),\\
0,\qquad \lambda\in (- i\lambda_{0}\tan(\frac{\pi}{12}),\pm i\lambda_{0}\tan(\frac{\pi}{12})).\\
\end{cases}
 \end{align}
\end{prop}
\begin{proof}
For $\lambda\in\Sigma\setminus(\mathcal{U}_{-\lambda_{0}}\cup\mathcal{U}_{\lambda_{0}})$, using the definition of $J^{E}$ \eqref{65}, we get
\begin{align}\label{67}
\left\vert J^{E}(\lambda)-\mathbb{I}\right\vert=\left\vert J^{(2)}-\mathbb{I}\right\vert.
 \end{align}
Thus, the estimates for $\left\vert J^{E}(\lambda)-\mathbb{I}\right\vert$ is consistent with Proposition \ref{prop2}.

For $\lambda\in \partial\mathcal{U}_{\pm\lambda_{0}}$, the variable $\xi=\sqrt{48\lambda_{0}t}(\lambda\pm\lambda_{0})$ tends to infinity as $t\rightarrow \infty$, and following  the asymptotic expansion
 \begin{align}\label{68}
M^{(\pm\lambda_{0})}(\xi)=\mathbb{I}+\frac{M^{(\pm\lambda_{0})}_{1}}{\xi}+O(\frac{1}{\xi^{2}}),\qquad \xi\rightarrow\infty,
 \end{align}
we have
 \begin{align}\label{69}
M^{(\pm\lambda_{0})}(\xi)=\mathbb{I}+\frac{M^{(\pm\lambda_{0})}_{1}}{\sqrt{48\lambda_{0}t}(\lambda\mp\lambda_{0})}+O(\frac{1}{t}\ln t),\qquad t\rightarrow\infty,
 \end{align}
 which leads to $\left\vert M^{(\pm\lambda_{0})}(\xi)-\mathbb{I} \right\vert=O(t^{-\frac{1}{2}})$. Using the definition of $J^{E}$ \eqref{65}, we get
\begin{align}\label{70}
\left\vert J^{E}(\lambda)-\mathbb{I}\right\vert=\left\vert M^{(\pm\lambda_{0})}-\mathbb{I}\right\vert=O(t^{-\frac{1}{2}}).
 \end{align}
\end{proof}
The estimates \eqref{66} denote that the $\left\vert J^{E}(\lambda)-\mathbb{I}\right\vert$ decays uniformly.
Thus, RH problem \ref{rhp7} is a small norm RH problem whose existence
and uniqueness are guaranteed by \cite{Li-WKI36,Li-WKI37}.

Next, the solution of RH problem \ref{rhp7} will be constructed via using Beals-Coifman theorem. Firstly, the jump matrix $J^{E}$ can be decomposed into
\begin{align}\label{71}
J^{E}=(b_{-})^{-1}b_{+},\quad b_{-}=\mathbb{I},\quad b_{+}=J^{E},
 \end{align}
and we have
\begin{align}\label{72}
&(\omega_{E})_{-}=\mathbb{I}-b_{-}=0,\quad (\omega_{E})_{+}=b_{+}-\mathbb{I}=J^{E}-\mathbb{I},\notag\\
&\omega_{E}=(\omega_{E})_{+}+(\omega_{E})_{-}=J^{E}-\mathbb{I},\notag\\
&C_{\omega_{E}}f=C_{-}(f(\omega_{E})_{+})+C_{+}(f(\omega_{E})_{-})=C_{-}(f(J^{E}-\mathbb{I})),
 \end{align}
where $C_{-}$ denotes the Cauchy projection operator, given by
\begin{align}\label{73}
C_{-}f(\lambda)=\lim_{\lambda'\rightarrow \lambda\in\Sigma^{E}}\frac{1}{2\pi i}\int_{\Sigma^{E}}\frac{f(s)}{s-\lambda'}\mathrm{d}s,
 \end{align}
and $\left\Vert C_{-}\right\Vert_{L^{2}}$ is a finite value. Then, the solution of RH problem \ref{rhp7} can be written as
\begin{align}\label{74}
E(\lambda)=\mathbb{I}+\frac{1}{2\pi i}\int_{\Sigma^{E}}\frac{\mu_{E}(s)(J^{E}-\mathbb{I})}{s-\lambda}\mathrm{d}s,
 \end{align}
where $\mu_{E}\in L^{2}(\Sigma^{E})$ admits $(1-C_{\omega_{E}})\mu_{E}=\mathbb{I}$.

From the Proposition \ref{prop3}, we get
\begin{align}\label{75}
\left\Vert J^{E}-\mathbb{I}\right\Vert_{L^{p}(\Sigma^{E})}=O(t^{-\frac{1}{2}}), \quad p\in[1,+\infty).
 \end{align}
Based on the properties of the
Cauchy projection operator $C_{-}$ and \eqref{75}, we have
\begin{align}\label{76}
&\left\Vert C_{\omega_{E}}\right\Vert_{L^{2}(\Sigma^{E})}\lesssim \left\Vert C_{-}\right\Vert_{L^{2}(\Sigma^{E})}\left\Vert J^{E}-\mathbb{I}\right\Vert_{L^{\infty}(\Sigma^{E})}\lesssim O(t^{-\frac{1}{2}}),\notag\\
&\left\Vert \mu_{E}-\mathbb{I}\right\Vert_{L^{2}(\Sigma^{E})}=\left\Vert C_{\omega_{E}}\mu_{E}\right\Vert_{L^{2}(\Sigma^{E})}=\left\Vert C_{-}\right\Vert_{L^{2}(\Sigma^{E})} \left\Vert\mu_{E}\right\Vert_{L^{2}(\Sigma^{E})}\left\Vert J^{E}-\mathbb{I}\right\Vert_{L^{\infty}(\Sigma^{E})}\lesssim O(t^{-\frac{1}{2}}).
 \end{align}
From the first formula of \eqref{76}, we know that $(1-C_{\omega_{E}})^{-1}$ is existent. Thus, the $\mu_{E}$ and $E(\lambda)$ are existent and
unique. This makes it reasonable to define $M^{RHP}$ in \eqref{54}.

Moreover, to recover the potential $q(y, t)$, we need
discuss the asymptotic expansion of $E(\lambda)$ as $\lambda\rightarrow 0$, given by
\begin{align}\label{77}
E(\lambda)=E_{0}+E_{1}\lambda+O(\lambda^{2}), \qquad \lambda\rightarrow 0,
 \end{align}
where
\begin{align}\label{78}
&E_{0}=I+\frac{1}{2\pi i}\int_{\Sigma^{E}}\frac{\mu_{E}(s)(J^{E}-\mathbb{I})}{s}\mathrm{d}s, \notag\\
&E_{1}=\frac{1}{2\pi i}\int_{\Sigma^{E}}\frac{\mu_{E}(s)(J^{E}-\mathbb{I})}{s^{2}}\mathrm{d}s.
 \end{align}
They admit the following asymptotic behavior  as $t\rightarrow\infty$
\begin{align}\label{79}
E_{0}&=\mathbb{I}+\frac{1}{2\pi i}\oint_{\partial\mathcal{U}_{\lambda_{0}}}\frac{J^{E}-\mathbb{I}}{s}\mathrm{d}s+\frac{1}{2\pi i}\oint_{\partial\mathcal{U}_{-\lambda_{0}}}\frac{J^{E}-\mathbb{I}}{s}\mathrm{d}s+O(t^{-1})\notag\\
&=\mathbb{I}+\frac{1}{2\pi i}\oint_{\partial\mathcal{U}_{\lambda_{0}}}\frac{M_{1}^{(\lambda_{0})}}{\sqrt{48\lambda_{0} t}s(s-\lambda_{0})}\mathrm{d}s+\frac{1}{2\pi i}\oint_{\partial\mathcal{U}_{-\lambda_{0}}}\frac{M_{1}^{(-\lambda_{0})}}{\sqrt{48\lambda_{0} t}s(s+\lambda_{0})}\mathrm{d}s+O(t^{-1}\ln t)\notag\\
&=\mathbb{I}+\frac{M_{1}^{(\lambda_{0})}}{\sqrt{48\lambda_{0} t}\lambda_{0}}-\frac{M_{1}^{(-\lambda_{0})}}{\sqrt{48\lambda_{0} t}\lambda_{0}}+O(t^{-1}\ln t), \notag\\
E_{1}&=\frac{1}{2\pi i}\oint_{\partial\mathcal{U}_{\lambda_{0}}}\frac{J^{E}-\mathbb{I}}{s^{2}}\mathrm{d}s+\frac{1}{2\pi i}\oint_{\partial\mathcal{U}_{-\lambda_{0}}}\frac{J^{E}-\mathbb{I}}{s^{2}}\mathrm{d}s+O(t^{-1})\notag\\
&= \frac{1}{2\pi i}\oint_{\partial\mathcal{U}_{\lambda_{0}}}\frac{M_{1}^{(\lambda_{0})}}{\sqrt{48\lambda_{0} t}s^{2}(s-\lambda_{0})}\mathrm{d}s+\frac{1}{2\pi i}\oint_{\partial\mathcal{U}_{-\lambda_{0}}}\frac{M_{1}^{(-\lambda_{0})}}{\sqrt{48\lambda_{0} t}s^{2}(s+\lambda_{0})}\mathrm{d}s+O(t^{-1}\ln t)\notag\\
&=\frac{M_{1}^{(\lambda_{0})}}{\sqrt{48\lambda_{0} t}\lambda_{0}^{2}}+\frac{M_{1}^{(-\lambda_{0})}}{\sqrt{48\lambda_{0} t}\lambda_{0}^{2}}+O(t^{-1}\ln t).
 \end{align}
Obviously, we also get
\begin{align}\label{80}
E_{0}^{-1}=\mathbb{I}+O(t^{-\frac{1}{2}}).
 \end{align}

\subsubsection{Pure $\bar{\partial}$-RH problem}
In this subsection, we mainly analyse the pure $\bar{\partial}$-problem with $\bar{\partial}R^{(2)}=0$ by defining
\begin{align}\label{81}
M^{(3)}(\lambda)=M^{(2)}(\lambda)\left(M^{RHP}(\lambda)\right)^{-1},
\end{align}
which is continuous and has no jumps in the complex plane.

\begin{rhp}\label{rhp8}
Find a matrix-valued function $M^{(3)}(\lambda)$ which satisfies:\\

$\bullet$ $M^{(3)}(\lambda)$  is continuous in  $\mathbb{C}\setminus (\mathbb{R}\cup\Sigma)$;\\

$\bullet$ $\bar{\partial}M^{(3)}(\lambda)=M^{(3)}(\lambda)W^{(3)}(\lambda), \lambda\in \mathbb{C}$, where
\begin{align}\label{82}
W^{(3)}(\lambda)=M^{RHP}(\lambda)\bar{\partial}R^{(2)}(\lambda)\left(M^{RHP}(\lambda)\right)^{-1};
\end{align}

$\bullet$ $M^{(3)}(\lambda)=\mathbb{I}+O(\frac{1}{\lambda})$ as $\lambda\rightarrow\infty$.
\end{rhp}

The pure $\bar{\partial}$-problem is solved by the following integral equation
\begin{align}\label{83}
M^{(3)}(\lambda)=\mathbb{I}-\frac{1}{\pi}\iint_{\mathbb{C}}\frac{M^{(3)}(s)W^{(3)}(s)}{s-\lambda}\mathrm{d}A(s),
\end{align}
where $\mathrm{d}A(s)$ is the Lebesgue measure. It can be  written into the operator form
\begin{align}\label{84}
(\mathbb{I}-\mathcal{S})M^{(3)}(\lambda)=\mathbb{I},
\end{align}
of which $\mathcal{S}$ is the Cauchy operator
\begin{align}\label{84}
\mathcal{S}[f](\lambda)=-\frac{1}{\pi}\iint_{\mathbb{C}}\frac{f(s)W^{(3)}(s)}{s-\lambda}\mathrm{d}A(s).
\end{align}

\begin{prop}\label{prop4}
For large time $t$,
\begin{align}\label{85}
\left\Vert \mathcal{S}\right\Vert_{L^{\infty}\rightarrow L^{\infty}}\lesssim t^{-\frac{1}{4}},
\end{align}
which implies that the operator $\mathbb{I}-\mathcal{S}$ is invertible and the solution of pure $\bar{\partial}$-problem exists and is
unique.
\end{prop}
\begin{proof}
We mainly prove the case of $\lambda\in\Omega_{1}^{+}$, the other case is similar. Supposing that $f\in L^{\infty}(\Omega_{1}^{+})$, we derive the following inequality
\begin{align}\label{86}
\left\Vert\mathcal{S}(f)\right\Vert&\lesssim \frac{1}{\pi}\iint_{\Omega_{1}^{+}}\frac{\left\vert fM^{RHP}\bar{\partial}R^{(2)}(M^{RHP})^{-1}\right\vert}{\left\vert s-\lambda\right\vert}\mathrm{d}A(s)\notag\\
&\lesssim \iint_{\Omega_{1}^{+}}\frac{\left\vert \bar{\partial}R_{3}^{+}\right\vert e^{Re(2it\theta(s))}}{\left\vert s-\lambda\right\vert}\mathrm{d}A(s)\lesssim I_{1}+I_{2},
 \end{align}
where
\begin{align}\label{87}
I_{1}=\iint_{\Omega_{1}^{+}}\frac{\left\vert (p_{3}^{+})'(Re s)\right\vert e^{Re(2it\theta(s))}}{\left\vert s-\lambda\right\vert}\mathrm{d}A(s),\ I_{2}=\iint_{\Omega_{1}^{+}}\frac{\left\vert s-\lambda_{0}\right\vert^{-\frac{1}{2}} e^{Re(2it\theta(s))}}{\left\vert s-\lambda\right\vert}\mathrm{d}A(s).
 \end{align}
Taking $s=u+iv, \lambda=x+iy$, then $Re(2it\theta(s))=8tv(3\lambda_{0}^{2}-3u^{2}+v^{2})$, we have
\begin{align}\label{88}
I_{1}&=\int_{0}^{+\infty}\int_{\lambda_{0}+v}^{+\infty}\frac{\left\vert (p_{3}^{+})'(u)\right\vert e^{8tv(3\lambda_{0}^{2}-3u^{2}+v^{2})}}{\left\vert s-\lambda\right\vert}\mathrm{d}u\mathrm{d}v\notag\\
&=\int_{0}^{+\infty}\int_{v}^{+\infty}\frac{\left\vert (p_{3}^{+})'(u)\right\vert e^{8tv(3\lambda_{0}^{2}-3(u+\lambda_{0})^{2}+v^{2})}}{\left\vert s-\lambda\right\vert}\mathrm{d}u\mathrm{d}v\notag\\
&\lesssim \int_{0}^{+\infty}\int_{v}^{+\infty}\frac{\left\vert (p_{3}^{+})'(u)\right\vert e^{8tv(3\lambda_{0}^{2}-3(v+\lambda_{0})^{2}+v^{2})}}{\left\vert s-\lambda\right\vert}\mathrm{d}u\mathrm{d}v\notag\\
&\lesssim \int_{0}^{+\infty}e^{-16tv(3v\lambda_{0}+v^{2})}\int_{v}^{+\infty}\frac{\left\vert (p_{3}^{+})'(u)\right\vert}{\left\vert s-\lambda\right\vert}\mathrm{d}u\mathrm{d}v\notag\\
&\lesssim \int_{0}^{+\infty}e^{-16tv(3v\lambda_{0}+v^{2})}\left\Vert (p_{3}^{+})'\right\Vert_{L^{2}}\left\Vert \frac{1}{s-\lambda}\right\Vert_{L^{2}} \mathrm{d}v\notag\\
&\lesssim \int_{0}^{+\infty}\frac{e^{-16tv(3v\lambda_{0}+v^{2})}}{\sqrt{\mid y-v\mid}} \mathrm{d}v\notag\\
&\lesssim \int_{0}^{y}\frac{e^{-48tv^{2}\lambda_{0}}}{\sqrt{ y-v}} \mathrm{d}v+\int_{y}^{+\infty}\frac{e^{-48tv^{2}\lambda_{0}}}{\sqrt{v-y}} \mathrm{d}v\notag\\
&\lesssim (t)^{-1/4}\int_{0}^{1}\frac{1}{\sqrt{\omega(1-\omega)}} \mathrm{d}\omega+(t)^{-1/4}\int_{0}^{\infty}\frac{e^{-48\lambda^{2}}}{\sqrt{\lambda}} \mathrm{d}\omega\notag\\
&\lesssim t^{-1/4}.
\end{align}
\begin{align}\label{89}
I_{2}&=\int_{0}^{+\infty}\int_{\lambda_{0}+v}^{+\infty}\frac{((u-\lambda_{0})^{2}+v^{2})^{-1/4} e^{8tv(3\lambda_{0}^{2}-3u^{2}+v^{2})}}{\left\vert s-\lambda\right\vert}\mathrm{d}u\mathrm{d}v\notag\\
&\lesssim \int_{0}^{+\infty}e^{-16tv(3v\lambda_{0}+v^{2})}\int_{v}^{+\infty}\frac{1}{\left\vert s-\lambda\right\vert \left\vert s-\lambda_{0}\right\vert^{1/2}}\mathrm{d}u\mathrm{d}v\notag\\
&\lesssim \int_{0}^{+\infty}e^{-48tv^{2}\lambda_{0}}\left\Vert\frac{1}{\left\vert s-\lambda\right\vert} \right\Vert_{L^{q}}\left\Vert\frac{1}{\left\vert s-\lambda_{0}\right\vert^{1/2}}\right\Vert_{L^{p}} \mathrm{d}v \notag\\
&\lesssim \int_{0}^{+\infty}e^{-48tv^{2}\lambda_{0}}v^{1/p-1/2}\left\vert v-y\right\vert^{1/q-1} \mathrm{d}v \notag\\
&= \int_{0}^{y}e^{-48tv^{2}\lambda_{0}}v^{1/p-1/2}\left\vert v-y\right\vert^{1/q-1} \mathrm{d}v+\int_{y}^{+\infty}e^{-48tv^{2}\lambda_{0}}v^{1/p-1/2}\left\vert v-y\right\vert^{1/q-1} \mathrm{d}v \notag\\
&\lesssim \int_{0}^{1}\sqrt{y}e^{-192ty^{2}\omega^{2}\lambda_{0}}\omega^{1/p-1/2}\left\vert 1-\omega\right\vert^{1/q-1} \mathrm{d}\omega+
(\lambda_{0}t)^{-1/4}\int_{0}^{+\infty}\eta^{-1/2}e^{-\eta}\mathrm{d}\eta\notag\\
&\lesssim t^{-1/4}.
 \end{align}
\end{proof}

Similarly, as $\lambda\rightarrow 0$, we extend $M^{(3)}(\lambda)$ into following form
\begin{align}\label{90}
M^{(3)}(\lambda)=\mathbb{I}+M^{(3)}_{0}+M^{(3)}_{1}\lambda+O(\lambda^{2}), \quad \lambda\rightarrow 0,
 \end{align}
\begin{align}\label{91}
&M^{(3)}_{0}=-\frac{1}{\pi}\iint_{\mathbb{C}}\frac{M^{(3)}(s)W^{(3)}(s)}{s}\mathrm{d}A(s),\notag\\
&M^{(3)}_{1}=-\frac{1}{\pi}\iint_{\mathbb{C}}\frac{M^{(3)}(s)W^{(3)}(s)}{s^{2}}\mathrm{d}A(s).
\end{align}

\begin{prop}\label{prop5}
For large time $t$, $M^{(3)}_{0}$ and $M^{(3)}_{1}$ admit the following inequality
\begin{align}\label{92}
\left\vert M^{(3)}_{0}\right\vert \lesssim t^{-\frac{3}{4}},\qquad \left\vert M^{(3)}_{1}\right\vert \lesssim t^{-\frac{3}{4}}.
 \end{align}
\end{prop}

\begin{proof}
We mainly focus on the case of $\lambda\in\Omega_{1}^{+}$, the other case is similar.
\begin{align}\label{93}
\left\vert M^{(3)}_{0}\right\vert &\lesssim \frac{1}{\pi}\iint_{\Omega_{1}^{+}}\frac{\left\vert M^{(3)}(s) M^{RHP}(s)\bar{\partial}R^{(2)}(s)\left(M^{RHP}(s)\right)^{-1}\right\vert}{\left\vert s\right\vert}\mathrm{d}A(s)\notag\\
&\lesssim \frac{1}{\pi}\left\Vert M^{(3)}(s)\right\Vert_{L^{\infty}(\Omega_{1}^{+})} \left\Vert M^{RHP}(s) \right\Vert_{L^{\infty}(\Omega_{1}^{+})} \left\Vert \left(M^{RHP}(s)\right)^{-1}\right\Vert_{L^{\infty}(\Omega_{1}^{+})}\iint_{\Omega_{1}^{+}}\frac{\left\vert \bar{\partial}R_{3}^{+}e^{2it\theta}\right\vert}{\left\vert s\right\vert}\mathrm{d}A(s)\notag\\
&\lesssim \iint_{\Omega_{1}^{+}}\frac{\left\vert (p_{3}^{+})'\right\vert e^{8tv(3\lambda_{0}^{2}-3u^{2}+v^{2})}}{\left\vert s\right\vert}\mathrm{d}A(s)+\iint_{\Omega_{1}^{+}}\frac{e^{8tv(3\lambda_{0}^{2}-3u^{2}+v^{2})}}{\left\vert s\right\vert\left\vert s-\lambda_{0}\right\vert^{1/2}}\mathrm{d}A(s)\notag\\
&\lesssim I_{3}+I_{4}.
\end{align}
\begin{align}\label{94}
I_{3}&=\iint_{\Omega_{1}^{+}}\frac{\left\vert (p_{3}^{+})'\right\vert e^{8tv(3\lambda_{0}^{2}-3u^{2}+v^{2})}}{\left\vert s\right\vert}\mathrm{d}A(s)=\int_{0}^{+\infty}\int_{\lambda_{0}+v}^{+\infty}\frac{\left\vert (p_{3}^{+})'\right\vert e^{8tv(3\lambda_{0}^{2}-3u^{2}+v^{2})}}{\sqrt{u^{2}+v^{2}}}\mathrm{d}u\mathrm{d}v\notag\\
&\lesssim\int_{0}^{+\infty}\left\Vert (p_{3}^{+})'\right\Vert_{L^{2}}\left(\int_{\lambda_{0}+v}^{+\infty} \frac{e^{16tv(3\lambda_{0}^{2}-3u^{2}+v^{2})}}{u^{2}+v^{2}}\mathrm{d}u\right)^{1/2}\mathrm{d}v\notag\\
&\lesssim \int_{0}^{+\infty}\left(\int_{v}^{+\infty} \frac{e^{-96tv\lambda_{0}u}}{(u+\lambda_{0})^{2}+v^{2}}\mathrm{d}u\right)^{1/2}\mathrm{d}v\notag\\
&\lesssim \int_{0}^{+\infty} \frac{1}{\sqrt{(v+\lambda_{0})^{2}+v^{2}}}\left(\int_{v}^{+\infty} e^{-96tv\lambda_{0}u}\mathrm{d}u\right)^{1/2}\mathrm{d}v\notag\\
&\lesssim \int_{0}^{+\infty} \left(\int_{v}^{+\infty} e^{-96tv\lambda_{0}u}\mathrm{d}u\right)^{1/2}\mathrm{d}v\notag\\
&\lesssim t^{-1/2}\int_{0}^{+\infty} \frac{1}{\sqrt{v}}e^{-48tv^{2}\lambda_{0}}\mathrm{d}v\notag\\
&\lesssim t^{-3/4}.
\end{align}
\begin{align}\label{95}
I_{4}&=\iint_{\Omega_{1}^{+}}\frac{e^{8tv(3\lambda_{0}^{2}-3u^{2}+v^{2})}}{\left\vert s\right\vert\left\vert s-\lambda_{0}\right\vert^{1/2}}\mathrm{d}A(s)=\int_{0}^{+\infty}\int_{\lambda_{0}+v}^{+\infty}\frac{e^{8tv(3\lambda_{0}^{2}-3u^{2}+v^{2})}}{\left\vert s\right\vert\left\vert s-\lambda_{0}\right\vert^{1/2}}\mathrm{d}u\mathrm{d}v\notag\\
&\lesssim\int_{0}^{+\infty}\left\Vert\frac{1}{\sqrt{\left\vert s-\lambda_{0}\right\vert}}\right\Vert_{L^{p}}
\left(\int_{v}^{+\infty}\frac{e^{-48qtvu\lambda_{0}}}{((u+\lambda_{0})^{2}+v^{2})^{q/2}}\mathrm{d}u\right)^{1/q}\mathrm{d}v\notag\\
&\lesssim\int_{0}^{+\infty}v^{1/p-1/2}\frac{1}{((v+\lambda_{0})^{2}+v^{2})^{q/2}}
\left(\int_{v}^{+\infty}e^{-48qtvu\lambda_{0}}\mathrm{d}u\right)^{1/q}\mathrm{d}v\notag\\
&\lesssim\int_{0}^{+\infty}v^{1/p-1/2}\left(\int_{v}^{+\infty}e^{-48qtvu\lambda_{0}}\mathrm{d}u\right)^{1/q}\mathrm{d}v\notag\\
&\lesssim\int_{0}^{+\infty}v^{1/p-1/2}(qtv)^{-1/q}e^{-48tv^{2}\lambda_{0}}\mathrm{d}v\notag\\
&\lesssim t^{-1/q}\int_{0}^{+\infty}v^{2/p-3/2}e^{-48tv^{2}\lambda_{0}}\mathrm{d}v\notag\\
&\lesssim t^{-3/4}\int_{0}^{+\infty}\omega^{2/p-3/2}e^{-48\omega^{2}}\mathrm{d}v\notag\\
&\lesssim t^{-3/4}.
\end{align}
Then, we have $\left\vert M^{(3)}_{0}\right\vert \lesssim t^{-\frac{3}{4}}$. Furthermore, since $(u^{2}+v^{2})^{-1/2}\leq\frac{1}{\lambda_{0}}$, we arrive at $\left\vert M^{(3)}_{1}\right\vert \lesssim \frac{\left\vert M^{(3)}_{0}\right\vert}{\lambda_{0}} \lesssim t^{-\frac{3}{4}}$.
\end{proof}

\subsection{The final step}
According to a series of transformation including \eqref{49}, \eqref{54} and \eqref{81}, we obtain
\begin{align}\label{97}
M^{(1)}(\lambda)=M^{(3)}(\lambda)E(\lambda)(R^{(2)}(\lambda))^{-1},\quad \lambda\in \mathbb{C}\backslash (\mathcal{U}_{\lambda_{0}}\cup\mathcal{U}_{-\lambda_{0}}).
\end{align}
Taking $\lambda\rightarrow 0$ in the above equality, we get
\begin{align}\label{98}
M_{0}^{(1)}=\left(\mathbb{I}+M^{(3)}_{0}\right)E_{0},\qquad M_{1}^{(1)}=\left(\mathbb{I}+M^{(3)}_{0}\right)E_{1}+M^{(3)}_{1}E_{0},
\end{align}
which leads to
\begin{align}\label{99}
(M_{0}^{(1)})^{-1}M_{1}^{(1)}=\left(\left(\mathbb{I}+M^{(3)}_{0}\right)E_{0}\right)^{-1}\left(\left(\mathbb{I}+M^{(3)}_{0}\right)E_{1}+M^{(3)}_{1}E_{0}\right)=E_{1}+O(t^{-3/4}).
\end{align}
Combining the reconstruction formula  \eqref{45} and formulas \eqref{64.1}, \eqref{79}, we  immediately obtain
\begin{align}\label{100}
&q(x,t)=q(y,t)=2\lambda_{0}^{-2}(48t\lambda_{0})^{-\frac{1}{2}}Im[M_{1}^{(\lambda_{0})}]_{12}+O(t^{-3/4}),\notag\\
&c_{+}=2\lambda_{0}^{-2}(48t\lambda_{0})^{-\frac{1}{2}}Im[M_{1}^{(\lambda_{0})}]_{11}+i\delta_{1}+O(t^{-3/4}).
\end{align}
According to \eqref{A9}, we finally derive
\begin{align}\label{101}
q(x,t)=\sqrt{\frac{\nu(\hat{\lambda}_{0})}{12t\hat{\lambda}_{0}^{5}}}\sin\left(16t\hat{\lambda}_{0}^{3}+\nu(\hat{\lambda}_{0})\ln(48t\hat{\lambda}_{0})
+\Theta\right)+O(t^{-3/4}),
\end{align}
where
\begin{align}\label{102}
&\Theta=-\frac{5}{4}\pi-arg\left(\Gamma(i\nu(\hat{\lambda}_{0}))\right)-arg\left(\bar{r}(\hat{\lambda}_{0})\right)
+2\left(\int_{-\infty}^{-\hat{\lambda}_{0}}+\int_{\hat{\lambda}_{0}}^{+\infty}\right)\ln\left\vert s-\hat{\lambda}_{0}\right\vert \mathrm{d}\nu(s)
+2i\hat{\lambda}_{0}\delta_{1},\notag\\
&\nu(s)=-\frac{1}{2\pi}\ln\left(1+\left\vert r(s)\right\vert^{2}\right),\quad \hat{\lambda}_{0}=\sqrt{\frac{x}{12t}}.
\end{align}

\begin{appendices}
\section{}\label{A}
In this part, we aim to solve the model RH problem and provide $[M_{1}^{(\lambda_{0})}]_{12}$ explicitly via
introducing the following transformation(see Figure \ref{F3})
\begin{align}\label{A1}
M^{mod}=M^{(\lambda_{0})}G_{j},\qquad  \xi \in \Omega_{j},\qquad j=0, \cdots, 4,
\end{align}
where
\begin{gather}\label{A2}
G_{0}=e^{-\frac{1}{4}i\xi^{2}\sigma_{3}}\xi^{-i\nu(\lambda_{0})\sigma_{3}},\notag\\
G_{1}=G_{0}\left(\begin{array}{cc}
    1  &  0\\
    \frac{\overline{r_{0}}}{1+\left\vert r_{0}\right\vert^{2}}  &  1\\
\end{array}\right),\quad G_{2}=G_{0}\left(\begin{array}{cc}
    1  &  -\frac{r_{0}}{1+\left\vert r_{0}\right\vert^{2}}\\
    0  &  1\\
\end{array}\right),\notag\\
G_{3}=G_{0}\left(\begin{array}{cc}
    1  &  r_{0}\\
    0  &  1\\
\end{array}\right),\quad G_{4}=G_{0}\left(\begin{array}{cc}
    1  &  0\\
    -\overline{r_{0}}  &  1\\
\end{array}\right).
\end{gather}
The above transformation leads to a model RH problem for $M^{mod}$.
\begin{rhp}\label{rhp9}
 Find an analysis function $M^{mod}(\xi)$ which satisfies:\\

 $\bullet$ $M^{mod}(\xi)$ is meromorphic in $\mathbb{C}\setminus \mathbb{R}$;\\

 $\bullet$ $M_{+}^{mod}(\xi)=M_{-}^{mod}(\xi)J^{mod}(\lambda_{0}), \xi\in \mathbb{R}$, where
\begin{align}\label{A3}
J^{mod}(\lambda_{0})=\left(\begin{array}{cc}
    1  &  r_{0}\\
  \overline{r_{0}} &  1+\left\vert r_{0}\right\vert^{2}\\
\end{array}\right);
\end{align}

$\bullet$ $M^{mod}(\xi)\rightarrow e^{-\frac{1}{4}i\xi^{2}\sigma_{3}}\xi^{-i\nu(\lambda_{0})\sigma_{3}}$ as $\xi\rightarrow\infty$.
\end{rhp}

This RH problem can be solved unambiguously through using the Liouville's theorem and parabolic cylinder equation. Since the jump matrix $J^{mod}$  is constant, the following equation can be derived
\begin{align}\label{A4}
\frac{\mathrm{d}}{\mathrm{d}\xi}M^{mod}+\left(\begin{array}{cc}
    \frac{i}{2}\xi  &  \beta\\
  \alpha &  -\frac{i}{2}\xi\\
\end{array}\right)M^{mod}=0,
\end{align}
where $\beta=-i\left[M_{1}^{(\lambda_{0})}\right]_{12}, \alpha=i\left[M_{1}^{(\lambda_{0})}\right]_{21}$.
Its solution is
\begin{align}\label{A5}
M^{mod}=\left(\begin{array}{cc}
    M_{11}^{mod}  &  \frac{\frac{i}{2}\xi M_{22}^{mod}-\frac{\mathrm{d}M_{22}^{mod}}{\mathrm{d}\xi}}{\alpha}\\
  \frac{\frac{i}{2}\xi M_{11}^{mod}+\frac{\mathrm{d}M_{11}^{mod}}{\mathrm{d}\xi}}{-\beta} &  M_{22}^{mod}\\
\end{array}\right),
\end{align}
of which $M_{jj}^{mod}, j=1, 2,$ satisfy the standard parabolic cylinder equation
\begin{align}\label{A6}
&\frac{\mathrm{d}^{2}}{\mathrm{d}\xi^{2}}M_{11}^{mod}+(\frac{i}{2}-\alpha\beta+\frac{\xi^{2}}{4})M_{11}^{mod}=0,\notag\\
&\frac{\mathrm{d}^{2}}{\mathrm{d}\xi^{2}}M_{22}^{mod}+(-\frac{i}{2}-\alpha\beta+\frac{\xi^{2}}{4})M_{22}^{mod}=0.
\end{align}
As $\xi\rightarrow\infty$, we have $M_{11}^{mod}\rightarrow e^{-\frac{1}{4}i\xi^{2}}\xi^{-i\nu(\lambda_{0})},$ $ M_{22}^{mod}\rightarrow e^{\frac{1}{4}i\xi^{2}}\xi^{i\nu(\lambda_{0})}$, which results in
\begin{align}\label{A7}
M_{11}^{mod}=\left\{
\begin{array}{lr}
(e^{-\frac{3\pi}{4}i})^{i\nu(\lambda_{0})}D_{-i\nu(\lambda_{0})}(e^{-\frac{3\pi}{4}i}\xi) \qquad \mbox{Im}(\xi)<0,\\
\\
(e^{\frac{\pi}{4}i})^{i\nu(\lambda_{0})}D_{-i\nu(\lambda_{0})}(e^{\frac{\pi}{4}i}\xi) \qquad \qquad \mbox{Im}(\xi)>0,
  \end{array}
\right.
\end{align}
\begin{align}\label{A8}
M_{22}^{mod}=\left\{
\begin{array}{lr}
(e^{-\frac{\pi}{4}i})^{-i\nu(\lambda_{0})}D_{i\nu(\lambda_{0})}(e^{-\frac{\pi}{4}i}\xi) \ \qquad \mbox{Im}(\xi)<0,\\
\\
(e^{\frac{3\pi}{4}i})^{-i\nu(\lambda_{0})}D_{i\nu(\lambda_{0})}(e^{\frac{3\pi}{4}i}\xi) \quad \qquad \mbox{Im}(\xi)>0.
  \end{array}
\right.
\end{align}
Then, one gets
\begin{gather}
M^{mod}_{-}(\xi)^{-1}M^{mod}_{+}(\xi)=M^{mod}_{-}(0)^{-1}M^{mod}_{+}(0)=\notag\\
\left(\begin{array}{cc}
  (e^{-\frac{3\pi}{4}i})^{i\nu}\frac{2^{\frac{-i\nu}{2}}\sqrt{\pi}}{\Gamma(\frac{1+i\nu}{2})}    & e^{-\frac{\pi}{4}i}(e^{-\frac{\pi}{4}i})^{-i\nu}\frac{2^{\frac{1+i\nu}{2}}\sqrt{\pi}}{\alpha\Gamma(\frac{-i\nu}{2})} \\
  e^{-\frac{3\pi}{4}i}(e^{-\frac{3\pi}{4}i})^{i\nu}\frac{2^{\frac{1-i\nu}{2}}\sqrt{\pi}}{\beta\Gamma(\frac{i\nu}{2})} &   (e^{-\frac{\pi}{4}i})^{-i\nu}\frac{2^{\frac{i\nu}{2}}\sqrt{\pi}}{\Gamma(\frac{1-i\nu}{2})}\\
\end{array}\right)^{-1}\notag\\
\left(\begin{array}{cc}
  (e^{\frac{\pi}{4}i})^{i\nu}\frac{2^{\frac{-i\nu}{2}}\sqrt{\pi}}{\Gamma(\frac{1+i\nu}{2})}    & e^{\frac{3\pi}{4}i}(e^{\frac{3\pi}{4}i})^{-i\nu}\frac{2^{\frac{1+i\nu}{2}}\sqrt{\pi}}{\alpha\Gamma(\frac{-i\nu}{2})} \\
  e^{\frac{\pi}{4}i}(e^{\frac{\pi}{4}i})^{i\nu}\frac{2^{\frac{1-i\nu}{2}}\sqrt{\pi}}{\beta\Gamma(\frac{i\nu}{2})} &   (e^{\frac{3\pi}{4}i})^{-i\nu}\frac{2^{\frac{i\nu}{2}}\sqrt{\pi}}{\Gamma(\frac{1-i\nu}{2})}\\
\end{array}\right)\notag\\
=\left(\begin{array}{cc}
    1  &  r_{0}\\
  \overline{r_{0}} &  1+\left\vert r_{0}\right\vert^{2}\\
\end{array}\right).
\end{gather}
Finally, the result comes to
\begin{align}\label{A9}
[M_{1}^{(\lambda_{0})}]_{12}=\frac{\sqrt{2\pi}e^{-\frac{3\pi i}{4}+\frac{\pi\nu(\lambda_{0}) }{2}}}{i\overline{r_{0}}\Gamma(i\nu(\lambda_{0}))},\quad [M_{1}^{(\lambda_{0})}]_{21}=\frac{i\sqrt{2\pi}e^{-\frac{\pi i}{4}+\frac{\pi\nu(\lambda_{0})}{2}}}{r_{0}\Gamma(-i\nu(\lambda_{0}))}.
\end{align}
\end{appendices}

\section{Acknowledgements}
This work was supported by the National Natural Science Foundation of China (No. 12175069 and No. 12235007), Science and Technology Commission of Shanghai Municipality (No. 21JC1402500 and No. 22DZ2229014) and Natural Science Foundation of Shanghai (No. 23ZR1418100).


\begin{thebibliography}{99}
\bibitem{Brunelli-JMP}
Brunelli J. C. The short pulse hierarchy. J  math. phys.,  46(12), 123507(2005).
\bibitem{Brunelli-JMP4}
Wadati M., Konno K., Ichikawa Y. H. New integrable nonlinear evolution equations. J. Phys. Soc. Jpn.,  47(5), 1698-1700(1979).
\bibitem{Brunelli-JMP5}
Ichikawa Y., Konno K., Wadati M. Nonlinear transverse oscillation of elastic beams under tension. J. Phys. Soc. Jpn.,  50(5), 1799-1802(1981).
\bibitem{Brunelli-BJP10}
Brunelli J. C. The bi-Hamiltonian structure of the short pulse equation. Phys. Lett. A,  353(6), 475-478(2006).
\bibitem{Lin-PD13}
Chou K.S.,  Qu C. Integrable equations arising from motions of plane curves, Physica D, 162,  9-33(2002).
\bibitem{Lin-PD14}
Qu C.,  Zhang D. The WKI model of type II arises from motion of curves in E3, J. Phys. Soc. Japan, 74,  2941-2944(2005).
\bibitem{Lin-PD15}
Kruskal M.  Nonlinear wave equations, Lecture Notes Phys. 38,  310-354(1975).
\bibitem{Lin-PD16}
Dmitrieva L. $N$-loop solitons and their link with the complex Harry Dym equation, J. Phys. A, 27  8197(1994).

\bibitem{Wang16}
Ma W. X. Riemann-Hilbert problems and N-soliton solutions for a coupled mKdV
system, J. Geom. Phys. 132,    45-54(2018).
\bibitem{Wang17}
 Ma W. X. Application of the Riemann-Hilbert approach to the multicomponent AKNS integrable hierarchies, Nonlinear Analysis: RWA. 47, 1-17(2018).

\bibitem{Wang18}
 Guo B. L.,   Ling L. M. Riemann-Hilbert approach and $N$-soliton formula for coupled
derivative Schr\"{o}dinger equation, J. Math. Phys. 53,    133-3966(2012).
\bibitem{Wang19}
 Yang J.  Nonlinear Waves in Integrable and Non-integrable Systems, Society for Industrial
and Applied Mathematics (2010).
\bibitem{Wang20}
Tian S. F.   Initial-boundary value problems for the general coupled nonlinear Schr\"{o}dinger
equation on the interval via the Fokas method, J. Differ. Equations, 262(1),    506-558(2017).
\bibitem{Wang21}
Tian S. F.  The mixed coupled nonlinear Schr\"{o}dinger equation on the half-line via the Fokas
method, Proc. R. Soc. Lond. A 472(2195),    20160588(2016).

\bibitem{MaWemXiu-1}
 Manakov S. V.   Nonlinear Fraunnhofer diffraction, Sov. Phys. JETP, 38,   693-696(1974).
\bibitem{MaWemXiu-2}
Ablowitz  M. J.,   Newell A. C. The decay of the continuous spectrum for solutions of the Korteweg-de Vries
equation, J. Math. Phys., 14,   1277-1284(1973).
\bibitem{MaWemXiu-3}
 Zakharov  V. E.,   Manakov S. V.  Asymptotic behavior of non-linear wave systems integrated by the inverse
scattering method, Sov. Phys. JETP, 44,   106-112(1976).
\bibitem{MaWemXiu-4}
 Ablowitz  M. J.,   Segur H. Asymptotic solutions of the Korteweg-de Vries equation, Stud. Appl. Math., 57,   13-44(1977).
\bibitem{MaWemXiu-5}
Segur  H.,   Ablowitz M. J.  Asymptotic solutions and conservation laws for the nonlinear Schr\"{o}dinger equation I, J. Math. Phys.,  17,  710-713(1973).
\bibitem{MaWemXiu-6}
 Its  A. R.  Asymptotics of solutions of the nonlinear Schr\"{o}dinger equation and isomonodromic deformations of systems of linear differential equations, Sov. Math. Dokl., 24,   452-456(1981).

\bibitem{Wang23}
 Deift P. A.,  Zhou X. A steepest descent method for oscillatory Riemann-Hilbert problems, Ann. Math. 137,   295-368(1993).

\bibitem{MaWemXiu-8}
Deift  P.,  Venakides  S.,   Zhou X. The collisionless shock region for the long-time behavior of solutions of the KdV equation, Commun. Pure Appl. Math.,  47,   199-206(1994).
\bibitem{MaWemXiu-9}
Deift  P.,   Zhou X.  Asymptotics for the Painlev\'{e} II equation, Commun. Pure Appl. Math.,  48,   277-337(1995).
\bibitem{MaWemXiu-10}
Deift  P.,   Venakides S.,   Zhou X. New results in small dispersion KdV by an extension of the steepest descent method for Riemann-Hilbert problems, Int. Math. Res. Not.,  1997,   286-299(1997).

\bibitem{Peng-17}
Kamvissis  S.  Long time behavior for the focusing nonlinear Schr\"{o}edinger equation with real spectral
singularities, Commun. Math. Phys., 180,  325-341(1996).
\bibitem{Peng-18}
Grunert K.,  Teschl G. Long-time asymptotics for the Korteweg-de Vries equation
via nonlinear steepest descent, Math. Phys. Anal. Geom. 12,   287-324(2009).
\bibitem{Peng-19}
Cheng P. J.,  Venakides S.,  Zhou X. Long-time asymptotics for the pure radiation
solution of the sine-Gordon equation, Commun. Partial Differential Equations 24,  1195-
1262 (1999).
\bibitem{Peng-20}
Boutet de Monvel A.,  Kostenko A.,  Shepelsky D.,  Teschl G. Long-time
asymptotics for the Camassa-Holm equation, SIAM J. Math. Anal. 41, 1559-1588 (2009).
\bibitem{Peng-21}
Xu J.,  Fan E. G. Long-time asymptotics for the Fokas-Lenells equation with decaying
initial value problem: without solitons, J. Differ. Equations, 259(3),   1098-1148(2015).
\bibitem{Peng-22}
Guo B. L.,  Liu N.,  Wang Y. F.  Long-time asymptotics for the Hirota equation on
the half-line, Nonlinear Anal., 174, 118-140 (2018).
\bibitem{Peng-23}
Wang D. S.,  Wang X. L.  Long-time asymptotics and the bright $N$-soliton solutions
of the Kundu-Eckhaus equation via the Riemann-Hilbert approach, Nonlinear Anal.:
RWA, 41,   334-361(2018).

\bibitem{Tian-wang24}
 McLaughlin K.T.R.,  Miller P.D. The $\bar{\partial}$ steepest descent method and the asymptotic behavior of
polynomials orthogonal on the unit circle with fixed and exponentially varying non-analytic weights, Int.
Math. Res. Not., Art. ID 48673 (2006).
\bibitem{Tian-wang25}
 McLaughlin K.T.R.,  Miller P.D. The $\bar{\partial}$ steepest descent method for orthogonal polynomials on
the real line with varying weights, Int. Math. Res. Not., IMRN, Art. ID 075 (2008).
\bibitem{Tian-wang26}
Dieng M.,  McLaughlin K.D.T. Long-time Asymptotics for the NLS equation via dbar methods,
arXiv:0805.2807.
\bibitem{Tian-wang27}
Cuccagna S.,  Jenkins R.  On asymptotic stability of N-solitons of the defocusing nonlinear Schr\"{o}dinger
equation, Comm. Math. Phys, 343, 921-969(2016).
\bibitem{Li-WKI42}
Borghese M., Jenkins R.  McLaughlin, K.T.R.: Long-time asymptotic behavior
of the focusing nonlinear Schr\"{o}dinger equation, Ann. I. H. Poincar\'{e} Anal. 35,
887-920 (2018).

\bibitem{Tian-wang28}
Giavedoni P. Long-time asymptotic analysis of the Korteweg-de Vries equation via the dbar steepest
descent method: the soliton region, Nonlinearity, 30(3), 1165-1181 (2017).

\bibitem{Tian-wang30}
Yang Y.L.,  Fan E.G. Soliton resolution for the short-pulse equation, J. Differ. Equ. , 280,
644-689 (2021).
\bibitem{Tian-wang31}
Cheng Q.Y.,  Fan E.G. Long-time asymptotics for the focusing Fokas-Lenells equation in the solitonic
region of space-time, J. Differ. Equ., 309, 883-948(2022).
\bibitem{Tian-wang32}
Yang Y.L., Fan E.G, On the long-time asymptotics of the modified Camassa-Holm equation in spacetime solitonic regions, Adv. Math., 402, 108340(2022).
\bibitem{Tian-wang33}
Li Z.Q.,  Tian S.F,  Yang J.J. Soliton Resolution for the Wadati-Konno-Ichikawa Equation with
Weighted Sobolev Initial Data, Ann. Henri Poincar\'{e}, (2022).
\bibitem{Tian-wang34}
Li Z. Q., Tian S. F., Yang J. J. On the soliton resolution and the asymptotic stability of N-soliton solution for the Wadati-Konno-Ichikawa equation with finite density initial data in space-time solitonic regions. Adv.  Math.  409, 108639 (2022)
\bibitem{Li-WKI55}
Xu, J., Fan, E.G., Long-time asymptotic behavior for the complex short pulse
equation. J. Differ. Equ. 269, 10322-10349 (2020)
\bibitem{Geng-PDLi-WKI55}
Wang K., Geng X., Chen M., Riemann-Hilbert approach and long-time asymptotics of the positive flow short-pulse equation. Physica D,  439: 133383(2022)
\bibitem{Li-WKI53}
Boutet de Monvel, A., Shepelsky, D., Zielinski, L., The short pulse equation by
a Riemann-Hilbert approach. Lett. Math. Phys. 107, 1345-1373 (2017)
\bibitem{Li-WKI54}
Boutet de Monvel, A., Shepelsky, D., Zielinski, L., The short-wave model for
the Camassa-Holm equation: a Riemann-Hilbert approach. Inverse Probl. 27,
105006 (2011)

\bibitem{Li-WKI57}
Boutet de Monvel, A., Shepelsky, D., Riemann-Hilbert approach for the
Camassa-Holm equation on the line. C. R. Math. 343, 627-632 (2006)
\bibitem{Li-WKI58}
Boutet de Monvel, A., Shepelsky, D., Riemann-Hilbert problem in the inverse
scattering for the Camassa-Holm equation on the line. Math. Sci. Res. Inst.
Publ. 55, 53-75 (2007)

\bibitem{Li-WKI36}
Deift, P., Zhou, X., Long-Time Behavior of the Non-Focusing Nonlinear
Schr\"{o}dinger Equation, a Case Study, Lectures in Mathematical Sciences. Graduate School of Mathematical Sciences, University of Tokyo (1994)
\bibitem{Li-WKI37}
Deift, P., Zhou, X., Long-time asymptotics for solutions of the NLS equation
with initial data in a weighted Sobolev space. Commun. Pure Appl. Math. 56(8),
1029-1077 (2003)

\end{thebibliography}
\end{document}